\documentclass[12pt]{amsart}
\usepackage{amssymb,amsmath}
\usepackage{geometry}    
\usepackage{mathrsfs}

\usepackage{graphicx}

\usepackage{vmargin}

\usepackage{tikz}

\usepackage{yfonts}

\setmargrb{1in}{1in}{1in}{1in}

\newtheorem{theorem}{Theorem}[section]
\newtheorem{lemma}[theorem]{Lemma}
\newtheorem{proposition}[theorem]{Proposition}
\newtheorem{corollary}{Corollary}
\theoremstyle{definition}
\newtheorem{definition}{Definition}
\newtheorem{remark}{Remark}

\newtheorem{conjecture}{Conjecture}
\newtheorem{problem}{Problem}

\newcommand{\scp}[1]{\langle#1\rangle}

\newcommand{\uxs}{\boldsymbol{\xi}^{\ast}}

\newcommand{\bg}{\boldsymbol{\Gamma}_{\mathcal{L}}}

\newcommand{\uz}{\boldsymbol{\zeta}}

\newcommand{\ux}{\boldsymbol{\xi}}

\newcommand{\om}{\omega}
\newcommand{\wo}{\widehat{\omega}}

\newcommand{\R}{\mathbb{R}}

\newcommand{\uu}{ \mathbf{u} }

\newcommand{\rr}{ \mathbf{r} }

\newcommand{\g}{ \mathbf{g} }
\newcommand{\h}{ \mathbf{h} }

\newcommand{\w}{ \mathbf{w} }
\newcommand{\vv}{ \mathbf{v} }
\newcommand{\pp}{ \mathbf{p} }
\newcommand{\qq}{ \mathbf{q} }

\newcommand{\xx}{ \boldsymbol{x}  }
\newcommand{\yy}{ \boldsymbol{y}  }

\newcommand{\Z}{ \mathbb{Z}  }

\newcommand{\bb}{\boldsymbol{b}}
\newcommand{\ee}{\boldsymbol{e}}
\newcommand{\0}{\boldsymbol{0}}

\begin{document}

\title[On the geometry of best approximations for a linear form]{Metrical results on the geometry of best approximations for a linear form}

\author{Johannes Schleischitz}


\thanks{Middle East Technical University, Northern Cyprus Campus, Kalkanli, G\"uzelyurt \\
    johannes@metu.edu.tr ; jschleischitz@outlook.com}


\begin{abstract}
	Consider the integer best approximations of a 
	linear form in $n\ge 2$ real variables. While it is well-known that
	any tail of this sequence always spans a lattice 
	of dimension at least three, Moshchevitin showed that this bound
	is sharp for any $n\ge 2$. In this paper, we determine the exact Hausdorff
	and packing dimension of the set where equality occurs, 
	in terms of $n$.
	Moreover, independently we show that there exist real vectors
	whose best approximations lie in a union of two two-dimensional sublattices of $\Z^{n+1}$. Our lattices jointly span a lattice of dimension three only, 
thereby leading to an alternative constructive proof of Moshchevitin's result. We determine the packing dimension and up to a small error term $O(n^{-1})$ also the Hausdorff dimension of the according set. 
Our method combines a new construction for a linear
form in two variables $n=2$ with a result by
Moshchevitin to amplify them. We further employ the recent variatonal principle and 
some of its consequences, as well as estimates for Hausdorff
and packing dimensions of Cartesian products and fibers.
Our method permits much freedom for the induced classical
exponents of approximation.
\end{abstract}

\maketitle

{\footnotesize{

{\em Keywords}: best approximation, Hausdorff dimension, packing dimension, variational principle \\
Math Subject Classification 2020: 11J06, 11J13}}



\section{Best approximations in small sublattices  } \label{intro}
Let $\ux=(\xi_1,\ldots, \xi_n)\in \R^n$ and for simplicity
consider the maximum norm 
on $\R^{n}$ denoted by $\Vert.\Vert$. For convenience we introduce the notation $\uxs=(\ux,1)\in\R^{n+1}$.
A classical topic in Diophantine approximation is to 
study small absolute values of a linear form
\begin{equation} \label{eq:ONE}
\qq\cdot \uxs= q_1 \xi_1 + \cdots + q_n \xi_n + q_{n+1},
\end{equation}
for non-zero integer vectors $\qq= (q_1,\ldots, q_{n+1})$. 
Here small means compared to $\Vert\hat\qq\Vert$, where $\hat{\xx}\in\R^n$ denotes the restriction of $\xx\in \mathbb{R}^{n+1}$
to its first $n$ coordinates.
We assume \eqref{eq:ONE} does not vanish for any 
integer vector $\qq\ne \0$,
and then call $\ux$ totally irrational.
We call a vector $\qq\in\Z^{n+1}$ 
a best approximation for $\ux$ if 
\[
|\qq\cdot \uxs|= 
\min_{0<\Vert\hat\bb\Vert\le \Vert\hat\qq\Vert} |\bb\cdot \uxs|,
\]
where the minimum is taken over all $\bb\in\Z^{n+1}$ with norm of $\hat\bb\ne 0$ at most $\Vert\hat\qq\Vert$.\footnote{ We follow the classical definition as for example in~\cite{ngm}, that indeed uses 
the norms of the restricted integer vectors without last coordinate (constant term). When omitting the ``hat'', even
though $\Vert \hat\qq\Vert \asymp_{\ux} \Vert \qq\Vert$, the concrete sequence may change in some cases. Our results are valid 
with either definition. }
They are unique up to sign for totally irrational $\ux$.
Considering the set of all best approximations with norms 
in increasing order gives rise to the sequence $\qq_j\in \mathbb{Z}^{n+1}$, $j\ge 1$, of best approximations
associated to $\ux$
with the properties
\[
1=\Vert \hat{\qq}_1\Vert< \Vert\hat\qq_2\Vert <\cdots,\qquad  |\qq_1\cdot \uxs|> |\qq_2\cdot \uxs|>\cdots.
\]
Let us adapt the notation $R(\ux)$ from~\cite{ngm} 
to denote the minimum integer $R$ so that
some tail of the best approximations $(\qq_j)_{j\ge j_0}$ 
lies in an $R$-dimensional sublattice $\mathcal{L}=\mathcal{L}(\ux)$ of $\Z^{n+1}$ that may depend on $\ux$.
It is well-known that for $n\ge 2$ and totally irrational $\ux\in\R^n$, we have $R(\ux)\ge 3$. See~\cite[Theorem~1.2]{mosh} for the claim with its proof sketched in the same paper~\cite[\S~1.3]{mosh}, and~\cite[Theorem~7]{ngm} for generalizations to a system of linear forms. On the
other hand,  
Moshchevitin~\cite{mosh} showed that the following sets are not empty. 

\begin{definition}
	For $n\ge 2$ an integer, let $\Gamma_n$ be the set of all totally irrational $\ux\in\R^n$ inducing $R(\ux)=3$.
\end{definition}

By the above observation,
the set $\R^2\setminus \Gamma_2$ is a countable union of rational affine hyperplanes, 
hence we may assume $n\ge 3$. 
More precisely, denoting by $\dim_H$ and $\dim_P$
the Hausdorff and packing dimension respectively,
Moshchevitin's refinements~\cite[Theorems~12 \& 13]{ngm} 
of his own result directly 
imply the following fact. 

\begin{theorem}[Moshchevitin]  \label{MT}
	We have $\dim_P(\Gamma_n)\ge\dim_H(\Gamma_n)\ge n-2$.
\end{theorem}

In our first new result, relying
on auxiliary results from~\cite{ngm} and~\cite{dfsu1, dfsu2},
we determine the exact Hausdorff and packing dimension of 
the sets $\Gamma_n$.

\begin{theorem}  \label{G}
	The Hausdorff dimensions of $\Gamma_n$ are given as
	\begin{equation} \label{eq:mm}
	\dim_H(\Gamma_3)=\frac{17-\sqrt{13} }{8} = 1.6743\ldots ,\qquad   \dim_H(\Gamma_n)=n-2+\frac{2}{n}, \quad (n\ge 4),
	\end{equation}
	and their packing dimensions as
	\begin{equation} \label{eq:grim}
	\dim_P(\Gamma_n)\;= \; n-1, \qquad n\ge 3.
	\end{equation}
\end{theorem}

The lower bounds follow relatively easily from combining observations from~\cite{dfsu1, dfsu2, ngm}, together with some metrical theory of Cartesian products. The upper bounds
require more work, especially the three-dimensional sublattice
in the definition of $\Gamma_n$ being arbitrary causes our proof to 
become technical. 
Define the uniform exponent of approximation
with respect to a linear form as
\[
\wo(\ux)= \sup \{ t>0: \limsup_{Q\to\infty} \min_{\0\ne \Vert \hat\bb\Vert\le Q} Q^t|\bb\cdot \uxs|<\infty  \},
\]
and the ordinary exponent of approximation
\[
\om(\ux)= \sup \{ t>0: \liminf_{Q\to\infty} \min_{\0\ne \Vert \hat\bb\Vert\le Q} Q^t|\bb\cdot \uxs|<\infty  \}.
\]
Then by Dirichlet's Theorem for any $\ux\in\R^n$ we have
\begin{equation} \label{eq:diri}
\om(\ux)\ge \wo(\ux)\ge n.
\end{equation}
Generalizing 
the proof strategy of Theorem~\ref{G}, we can obtain similar 
results on best approximations in sublattices of higher dimension $k$.
It turns out that the packing dimension of the accordingly defined sets does not increase up to $k=n$.

\begin{theorem} \label{N}
	Let $3\le k\le n$ be integers and
	let $\Gamma_{n,k}:=\{ \ux\in\R^n: R(\ux)\le k \}$. Then
	\begin{equation} \label{eq:rhS}
	\dim_P(\Gamma_{n,k})= n-k+1+ \dim_P(\{ \ux\in\R^{k-1}: \wo(\ux)\ge n \})=n-1.
	\end{equation}
	Moreover
		\begin{equation} \label{eq:RHs}
	\dim_H(\Gamma_{n,k})\ge  n-k+1+ \dim_H(\{ \ux\in\R^{k-1}: \wo(\ux)>n  \}), 
	\end{equation}
	and conversely
		\begin{equation} \label{eq:Rhs}
	\dim_H(\Gamma_{n,k})\le  n-k+1+ \dim_H(\{ \ux\in\R^{k-1}: \wo(\ux)\ge n  \}). 
	\end{equation}
\end{theorem}


We implicitly assume any $\ux\in \Gamma_{n,k}$ to be totally irrational.
For $k=3$ we have $\Gamma_{n,3}=\Gamma_n$ and Theorem~\ref{N} becomes Theorem~\ref{G} via formulas from~\cite{dfsu1, dfsu2} (see also~\cite{bcc}).
The precise value for the packing dimension
in \eqref{eq:rhS} is evaluated via~\cite[Theorems~3.8 \& 4.9]{dfsu2},
for $k\ge 4$ the right hand sides
in \eqref{eq:RHs}, \eqref{eq:Rhs} are unknown, see~\cite{dfsu1, dfsu2} for approximative results. We strongly conjecture that 
for any $k\le n$ they coincide,
thereby leading to equalities in \eqref{eq:RHs}, \eqref{eq:Rhs}, however this seems not yet settled. 
Clearly a sufficient condition is
that the map 
$w\longmapsto \dim_H(\{ \ux\in\R^{k-1}: \wo(\ux)\ge w  \})$
is continuous for $w>k-1$ (it is discontinuous at $w=k-1$, 
see~\cite{dfsu1, dfsu2}, but conjecturally this is the only discontinuity).

The main focus of this paper is to study
the following problem on best approximations for a linear form.

\begin{problem}  \label{p1}
	For $n\ge 2$, does there exist totally irrational $\ux\in \R^n$ so that some tail of best approximations $(\qq_j)_{j\ge j_0}$ lies in a finite union of two-dimensional sublattices of $\Z^{n+1}$?
	If so, determine the minimum possible number $N=N(n)$ of sublattices. Determine/Estimate the Hausdorff and packing 
	dimensions of these sets as functions of $n$.
\end{problem}

Clearly $N\ge 2$ for $n\ge 2$ by the fact $R(\ux)\ge 3$ for totally
irrational $\ux$ recalled above.

\begin{remark}  \label{reh}
	For simultaneous approximation, the answer is negative, as any hyperplane
	contains only finitely many best approximations as soon as $\ux$ 
	is totally irrational. Moreover, in either linear form or simultaneous approximation problem, clearly any one-dimensional subspace contains at most one best approximation vector
	for given $\ux$,
	if so its (up to sign) unique integer vector with coprime coefficients.
\end{remark}

\begin{remark}  \label{hire}
	As pointed out to the author in private communication by N.G. Moshchevitin, with some effort a positive answer to
	Problem~\ref{p1} (omitting the metrical aspects) 
	with the optimal constant $N=2$
	can be derived from the lemma in~\cite{kmo} and its proof.
    This lemma essentially states the following:
    Let $G\subseteq \Z^n$ be a set of integer vectors that is not contained in a set of the form $\ell \cup F$, where $\ell$ is a line and $F$ a finite set, in $\R^n$. Then there is totally irrational $\ux\in\R^n$ such that for integer vectors 
    within the set
    $G^{\ast}=\{(g, h): g=(g_1,\ldots,g_n)\in G, h\in\Z\}\subseteq \Z^{n+1}$ we find very small linear forms
   $|\qq\cdot \uxs|$ for certain $\qq\in G^{\ast}$ (implying
   $\hat\qq\in G$), occurring
    with some density that in particular admits to ask for $\wo(\ux)=\infty$.
    Taking $G$ the union of two non-collinear rational
    one-dimensional sublattices (lines) $\ell_1, \ell_2$ 
    of $\Z^n$, its embedding $G^{\ast}$ in $\Z^{n+1}$
    lies in
    the union of the two two-dimensional sublattices $\scp{\ell_1, \ee_{n+1}}_{\Z}$ and $\scp{\ell_2, \ee_{n+1}}_{\Z}$ of $\Z^{n+1}$.
    As pointed out to the author by N.G. Moshchevitin, with some cumbersome additional geometrical
    arguments one can guarantee that these integer points indeed form
    some tail of the best approximations associated to $\ux$.
    However, this is not explicitly carried out in the short note~\cite{kmo} (nor in later work) as it was not of relevance
    in that paper. In our alternative construction regarding Problem~\ref{p1} below (proof of Theorem~\ref{H}), we fix the lines
    $\ell_i=\scp{\ee_i}_{\Z}, i=1,2$ as the first two coordinate axes
    of $\R^n$ and provide an explicit, rather elementary argument of this fact for certain $\ux$, based on Minkowski's Second Convex Body Theorem. Moreover, we address the metrical problem.
\end{remark}

\section{On Problem~\ref{p1}}

\subsection{Main new results}
As indicated in Remark~\ref{hire},
we give an almost complete answer to Problem~\ref{p1}. Indeed,
we show that for arbitrary $n$ and the maximum norm 
the answer is positive, and the optimal bound
$N(n)=2$ can be reached for any $n\ge 2$. Moreover,
we may choose the two two-dimensonal sublattices so that they
span a lattice of dimension only three in $\Z^{n+1}$.
Thereby we recover Moshchevitin's result that $R(\ux)=3$ can be reached, i.e. $\Gamma_n\ne \emptyset$, with a new proof, 
that we consider easier than the original one from~\cite{mosh}.

To state our result in full generality, we need to introduce some notation.
Let us first define the two-dimensional sublattices of $\Z^{n+1}$ given by
\[
\mathcal{H}_i= \scp{\ee_i, \ee_{n+1}}_{\mathbb{Z}}, \qquad 1\le i\le n,
\]
i.e. the $\mathbb{Z}$-span of the $i$-th and the $(n+1)$-st canonical base vector
in $\R^{n+1}$. For the immediate concern of Problem~\ref{p1}, 
we will only need $\mathcal{H}_1$ and $\mathcal{H}_2$.
Let us define the following properties of $\ux\in\R^n$
with induced sequence of best approximations $(\qq_j)_{j\ge 1}$: 
	\begin{itemize}
	\item[(i)] $\ux$ is totally irrational
	\item[(ii)] Some tail
	$(\qq_j)_{j\ge j_0}$ lies in the union
	of the two-dimensional sublattices
	$\mathcal{H}_1$ and $\mathcal{H}_2$	of $\Z^{n+1}$. 
	\item[(iii)] Some tail $(\qq_j)_{j\ge j_0}$
	lies in the three-dimensional sublattice 
	\[
	\scp{\mathcal{H}_1\cup \mathcal{H}_2}_{\R}\cap \Z^{n+1}= \scp{\ee_1,\ee_2,\ee_{n+1}}_{\Z}
	\]
	of $\Z^{n+1}$. In particular $R(\ux)=3$.
	\item[(iv)] Large best approximations lie in $\mathcal{H}_1$ and $\mathcal{H}_2$ alternatingly, i.e. for any $j\ge j_0$ 
	\[
	\qq_{2j}\in \mathcal{H}_1, \quad \qq_{2j+1}\in \mathcal{H}_2 \qquad\quad \text{or} \qquad \quad \qq_{2j}\in \mathcal{H}_2, \quad \qq_{2j+1}\in \mathcal{H}_1.
	\]
	\item[(v)] For $j\ge j_0$, three consecutive best approximations
	$\qq_j, \qq_{j+1}, \qq_{j+2}$ are linearly independent, thus 
	$\scp{\qq_j, \qq_{j+1}, \qq_{j+2}}_{\R}\cap \Z^{n+1}$ has full dimension three in $\scp{\ee_1,\ee_2,\ee_{n+1}}_{\Z}$
	\item[(vi)] For $w\in[n,\infty]$, we have
	\[
	\wo(\ux)=w, \qquad \om(\ux)=w^2.
	\]
\end{itemize}

We comment on the conditions and their mutual relations in~\S~\ref{SS} below. We want to remark that we expect
Theorem~\ref{H} below to remain true when replacing $\ee_1$ in $\mathcal{H}_1$ and $\ee_2$ in
$\mathcal{H}_2$ by any pair of linearly independent integer vectors in the 
two-dimensional subspace of $\R^{n+1}$ defined by $x_3=\cdots=x_{n+1}=0$. This is supported by~\cite{kmo}, see Remark~\ref{hire} above. 

\begin{definition}
	Let $n\ge 2$ be an integer. Define $\Theta^n\subseteq \R^n$  
	as the set of $\ux\in\R^n$ for which conditions (i)-(v) hold.
	For $w\in [n,\infty]$,
	define further $\Theta_n(w)\subseteq \R^{n}$ as the set of $\ux\in\R^n$ for which (i)-(vi) hold. 
\end{definition}

  Clearly
\begin{equation} \label{eq:absaty}
\Gamma_n\supseteq \Theta^n\supseteq \bigcup_{w\in[n,\infty] } \Theta_n(w), \qquad n\ge 2.
\end{equation}
We can now finally state the following rather satisfactory 
partial answer to
Problem~\ref{p1}. 

\begin{theorem}  \label{H}
Let $n\ge 2$ and $w\in [n,\infty]$. Then 
\begin{equation} \label{eq:theclaim}
\dim_H(\Theta_n(w)) \ge n-2+\frac{1}{w^2+1}\;,\qquad\quad \dim_H(\Theta^n) > n-2+\frac{1}{n^2+1}. 
\end{equation}
Moreover 
\begin{equation} \label{eq:DREI}
\dim_P(\Theta^2) \ge 1,\qquad\quad  \dim_P(\Theta^n) = n-1, \quad (n\ge 3).
\end{equation}
In particular $\Theta^n\ne \emptyset$ and $N=2$ can be reached
in Problem~\ref{p1}.
Conversely, we have 
\begin{equation}  \label{eq:zzwei}
\dim_H(\Theta^2) \leq \frac{7}{5}
\end{equation}
and
\begin{equation}  \label{eq:uc}
\dim_H(\Theta^{3})\le \frac{17-\sqrt{13}}{8},\qquad\quad
\dim_H(\Theta^n) \leq n-2+\frac{2}{n} \quad (n\ge 4).
\end{equation}
\end{theorem}

In view of \eqref{eq:absaty}, estimates \eqref{eq:uc} are an obvious consequence of \eqref{eq:mm}
from Theorem~\ref{G}. The bound \eqref{eq:zzwei} can be improved
with some effort, see Remark~\ref{RR} below. On the other hand,
we are unable to provide a non-trivial upper bound for $\dim_P(\Theta^2)$. Comparing the bounds of Theorem~\ref{G} and
Theorem~\ref{H} yields the following corollary.

\begin{corollary}
	We have
	$\dim_H(\Gamma_n)- \dim_H(\Theta^n)=o(1)$ as $n\to\infty$.
	We have $\dim_P(\Gamma_n)=\dim_P(\Theta^n)$ for any $n\ge 3$. Moreover 
	$2=\dim_H(\Gamma_2)>\dim_H(\Theta^2)$.
\end{corollary}

The following problem remains open.

\begin{problem}
	For $n\ge 3$, is the inequality $\dim_H(\Gamma_n)\ge \dim_H(\Theta^n)$ strict? 
\end{problem}

An explicit lower bound slightly exceeding
$n-2+1/(n^2+1)$ for the Hausdorff dimension of $\Theta^n$
can be readily deduced from the proof of Theorem~\ref{H} below,
see formula \eqref{eq:taurin}. 
Inserting small $n$ in this strengthened bound \eqref{eq:taurin},
the decimal expansions start with
\begin{equation} \label{eq:werte}
\dim_H(\Theta^2)\ge 0.2023\ldots, \quad 
\dim_H(\Theta^{3})\ge 1.1009\ldots, \quad 
\dim_H(\Theta^4)\ge 2.0590\ldots.
\end{equation}
See also Remark~\ref{hhh} below on bounds for Hausdorff or packing dimension when additionally restricting
$\wo(\ux)$, complementing the left estimate of \eqref{eq:theclaim}.
Our method suggests the
following refinements of \eqref{eq:theclaim} (and \eqref{eq:taurin}).

\begin{conjecture}  \label{c1}
	For $n\ge 2$ and $w\in[n,\infty]$ we have
	\[
	\dim_H(\Theta_n(w)) \ge n-2+\frac{2}{w^2+1},\qquad \dim_H(\Theta^n) > n-2+\frac{2}{n^2+1}.
	\]
\end{conjecture} 

See Remark~\ref{r3} in~\S~\ref{s4} below for more details.
For small $n$, Conjecture~\ref{c1} would yield considerable improvements of \eqref{eq:werte}.
While the right inequality is again strict, asymptotically
our results suggest just a small improvement, with lower bound 
still of order $n-2+2/n^2-O(n^{-4})$.
We wonder if the special choice of $\mathcal{H}_1, \mathcal{H}_2$
in the sets $\Theta_n(w), \Theta^n$ are significant. This is constituted in the following more general problem.

\begin{problem} \label{co2}
	Do the lower bounds of Theorem~\ref{H} (possibly Conjecture~\ref{c1}) hold for the set of $\ux\in\R^n$ with the property that some tail of best approximations 
	$(\qq_j)_{j\ge j_0}$ lies in a union of any two fixed two-dimensional
	sublattices of $\Z^{n+1}$? Do the upper bounds of Theorem~\ref{H} (possibly Conjecture~\ref{c1}) hold for the set of $\ux\in\R^n$ with some tail $(\qq_j)_{j\ge j_0}$ in a union of any two (or finite?) two-dimensional sublattices of $\Z^{n+1}$, independent of $\ux$?
\end{problem}

It seems the method of~\S~\ref{genc} can be used to verify the
second part of Problem~\ref{co2} for the special case of the two
sublattices jointly spanning a three-dimensional lattice in $\Z^{n+1}$. This
may be a necessary and sufficient criterion for the lower bounds as well.

The main substance of Theorem~\ref{H} are the lower bounds.
In short, to prove \eqref{eq:theclaim},
we combine a new construction for $n=2$ with
a result by Moshchevitin~\cite[Theorem~12]{ngm}. It follows 
that a ``generic'' vector in $\R^{n-2}$, in sense of Lebesgue measure,
gives rise to vectors in $\Theta_n(w)$ by adding two more suitable real components. This further directly implies the lower bound $n-2$
for the Hausdorff
dimension of the sets $\Theta_n(w)$. With some refined argument and using metrical results by Sun~\cite{sun} and
by Das, Fishman, Simmons, Urba\'nski~\cite{dfsu1, dfsu2}
we find the stronger lower bounds in \eqref{eq:theclaim}.
The upper bound \eqref{eq:zzwei} not implied by Theorem~\ref{G} 
follows independently from a classical formula by Jarn\'ik~\cite{jarnik} and the theory of continued fractions.

By small modifications of the proof of Theorem~\ref{H}, we
can obtain best approximations ultimately lying in 
a union of $k$ two-dimensional sublattices of $\Z^{n+1}$,
that together span a $(k+1)$-dimensional space,
but in no smaller number of two-dimensional subspaces. 
We want to explicitly state
this generalization of the case $k=2$ of Theorem~\ref{H}, but avoid detailed metrical formulas for brevity.

\begin{theorem}  \label{thm3}
	Let $n\ge 3$. For any $2\le k\le n$, there exists a set
	of Hausdorff dimension strictly greater than $n-k$, consisting of $\ux\in\R^n$ with property
	(i) and
	\begin{itemize}
		\item[$(ii^{\ast})$] Some tail of the best approximation sequence $(\qq_j)_{j\ge j_0}$
		lies in the union 
		\[
		\bigcup_{i=1}^{k} \mathcal{H}_i= \bigcup_{i=1}^{k} \scp{\ee_i, \ee_{n+1}}_{\Z}
		\]
		of $k$ two-dimensional sublattices 
		of $\Z^{n+1}$ that together span the $(k+1)$-dimensional sublattice 
		\[
		\mathcal{L}_{n,k}=\scp{\ee_1,\ldots,\ee_{k},\ee_{n+1}}_{\Z}\subseteq \Z^{n+1}.
		\] 
		\item[(vii)] No tail $(\qq_j)_{j\ge j_1}$ lies in union
		of less than $k$ two-dimensional sublattices of $\Z^{n+1}$, 
		similarly no tail is contained
		in a sublattice of dimension $k$ or less. 
	\end{itemize}	
\end{theorem}

In fact analogues of all (i)-(vi) hold for the $\ux$ in Theorem~\ref{thm3}. For fixed $n$,
it is natural to expect that the Hausdorff dimension of the
set in Theorem~\ref{thm3} increases with $k$. This is not reflected in the claim. For $k=n$, the according set has full $n$-dimensional Lebesgue
measure (a rigorous argument for this follows from similar method in~\S~\ref{31} below). An immediate corollary of Theorem~\ref{thm3} reads as follows.

\begin{corollary} \label{B}
For any $n\ge 2$, any value
$R(\ux)\in\{3,4,\ldots,n+1\}$ is attained for some
totally irrational $\ux\in\R^n$, in fact they form a set of positive
Hausdorff dimension. 	
\end{corollary}

Probably Corollary~\ref{B} could be derived independently from \eqref{eq:RHs}, \eqref{eq:Rhs} upon determining (estimating) 
the involved Hausdorff dimensions.
For sake of completeness,
we end this section with estimating
the size of the set $\mathcal{Y}^{n,k}$ of $\ux\in\R^n$ inducing infinitely many best approximations in any finite union of $k$-dimensional sublattices $\mathcal{L}_1, \ldots, \mathcal{L}_{i(\ux)}$, depending on $\ux$, of $\Z^{n+1}$. The proof is not complicated.

\begin{theorem} \label{Y}
	Let $n\ge k\ge 2$ be integers
	and $\mathcal{Y}^{n,k}$ be as above. Then
	\[
	\dim_H(\mathcal{Y}^{n,k}) \leq n-1+ \frac{k}{n+1}.
	\] 
\end{theorem}

For $n=k=2$ the upper bound becomes $5/3$. One may compare this
with the smaller bound $7/5$ 
from \eqref{eq:zzwei} for the smaller set $\Theta^2\subseteq \mathcal{Y}^{2,2}$ dealing with a special collection of two-dimensional sublattices in which all but finitely many best approximations lie.

\begin{problem}
	Are the packing dimensions of $\mathcal{Y}^{n,k}$ full? 
	What if we instead
	require all large best approximations to lie in $\mathcal{L}_1, \ldots, \mathcal{L}_{i(\ux)}$?
\end{problem}

\subsection{On the conditions (i)-(vi)}  \label{SS}

Clearly $(iv)\Rightarrow (ii)\Rightarrow (iii)$ and $(iv)\Rightarrow (i)$. Moreover $(iv)\Rightarrow (v)$ by the observation
on one-dimensional subspaces from Remark~\ref{reh}. So (i), (ii), (iii) and (v) are rather stated for sake of completeness.
Moreover, the theory of continued fractions 
and Dirichlet's Theorem 
\eqref{eq:diri} easily imply
that conversely $(ii)\Rightarrow (iv)$, see the proof in~\S~\ref{sieben}. So $(ii)\Leftrightarrow (iv)$.
We want to comment on
(vi) in the light of Theorem~\ref{MT}. The subset of vectors 
within $\Gamma_n$
originally constructed by Moshchevitin in~\cite{mosh} have the property
$\wo(\ux)=\infty$. Thus they
form a set of Hausdorff dimension at most $n-2$ in view of~\cite[Theorem~3.6]{dfsu2}, so the metrical claim in Theorem~\ref{MT} cannot be improved in this way. 
However, using the refinements from~\cite{ngm} together with
some new idea for linear forms in two variables, 
we will find $\ux$ satisfying (i)-(v) and with  
finite uniform exponent of approximation. This enables us to
surpass this treshold value $n-2$ for the Hausdorff dimension
even for the smaller sets $\Theta^n\subseteq \Gamma_n$ in Theorem~\ref{H}.

\section{A structural result on $\Gamma_n$}

In the proof of Theorem~\ref{G}, we show 
via~\cite[Theorem~12]{ngm} the following: For any
$(\xi_1,\xi_2)\in\R^2$ inducing $\wo(\xi_1,\xi_2)>n$, there
is some set $F_n(\xi_1,\xi_2)\subseteq \R^{n-2}$ of full $(n-2)$-dimensional Lebesgue measure so that $(\xi_1,\ldots,\xi_n)\in \Gamma_n$ for any $(\xi_3,\ldots,\xi_n)\in F_n(\xi_1,\xi_2)$
(in fact $(\xi_1,\ldots,\xi_n)\in \widetilde{\Gamma}_n$
for the smaller set $\widetilde{\Gamma}_n$
with special choice of the three-dimensional lattice $\scp{\ee_1,\ee_2,\ee_{n+1}}_{\Z}$). The sets $F_n(\xi_1,\xi_2)$ 
are hereby implicitly derived from the convergence
part of the Borel-Cantelli Lemma within the proof in~\cite{ngm}.
In our last new result, we provide an
explicit set for $F_n$ independent of the choice of $\xi_1,\xi_2$
in terms of Diophantine properties, upon
increasing the lower bound on the uniform exponent to $3n-4$.

\begin{theorem} \label{neu}
	Let $n\ge 2$. For any vector $\ux=(\xi_1,\ldots,\xi_n)$ in the set
	\[
	\mathcal{T}_n:= \left\{ (\xi_1,\xi_2): \wo(\xi_1,\xi_2)>3n-4 \right\}\times \{ (\xi_3,\ldots,\xi_n): \om(\xi_3,\ldots,\xi_n)=n-2 \}\subseteq \R^n,
	\]
	the tail of best approximations lies in the three-dimensional sublattice $\mathcal{L}_{n,3}=\scp{\ee_1,\ee_2,\ee_{n+1}}_{\Z}$ of $\Z^{n+1}$.
	In particular
	\[
	\mathcal{T}_n\subseteq \Gamma_n.
	\]
	Moreover $\wo(\ux)=\wo(\xi_1, \xi_2)$ and $\om(\ux)=\om(\xi_1, \xi_2)$ hold for any $\ux\in \mathcal{T}_n$.
\end{theorem}

See also Remark~\ref{r9} below for refinements.
From Theorem~\ref{neu} together with metrical results from~\cite{dfsu1, dfsu2}, we get another proof for the lower bound
$n-1$ for the packing dimension of $\Gamma_n$. For its Hausdorff dimensions, using a classical result of Khintchine~\cite{khint} and again~\cite{dfsu1, dfsu2},
we can 
deduce from Theorem~\ref{neu} the lower bound $n-2+2/(3n-4)$ for $n\ge 3$, weaker
than the bound in Theorem~\ref{G}. 
In fact both claims again hold for $\widetilde{\Gamma}_n$.
While there are some similarities 
underying the fundamental ideas of Theorem~\ref{neu}
and~\cite[Theorem~12]{ngm}, the proofs differ considerably. 
Theorem~\ref{neu} uses Minkowski's Second Convex Body Theorem instead
of the Borel Cantelli Lemma.

The choice of the maximum norm in our new results is just for convienence,
we may establish the analogous result for a large class of 
norms, including all $p$-norms $(\sum |x_i|^p)^{1/p}$, by small
modifications of the proof. In fact, at least in Theorems~\ref{G},~\ref{neu},
we may take any norm.

\section{Preliminary results}

\subsection{An auxiliary result by Moshchevitin}  \label{MO}

We recall a partial result of~\cite[Theorem~12]{ngm}.  
In the notation of~\cite{ngm}, its special case $n=1, m=2$ and $m^{\ast}$ equal to our present $n$, yields:

\begin{theorem}[Moshchevitin]  \label{moschus}
	Let $n\ge 3$. If $\qq_j=(q_{j,1}, q_{j,2}, q_{j,3})\in\Z^3$ is the sequence of best approximations for 
	$(\xi_1, \xi_2)\in\R^2$ and we have 
	\begin{equation}  \label{eq:mc}
	\sum_{j=1}^{\infty} \Vert \qq_{j+1}\Vert^{n} \log (\Vert \qq_{j+1}\Vert)\cdot \vert q_{j,1}\xi_1+q_{j,2}\xi_2+q_{j,3} |< \infty,
	\end{equation}
	then for almost all choices
	of remaining entries $(\xi_3,\ldots,\xi_n)\in\R^{n-2}$
	with respect to $(n-2)$-dimensional Lebesgue measure,
	for the vector $\ux=(\xi_1,\ldots,\xi_n)$ and some $j_0(\ux)$, the embedded sequence 
	\[
	q_{j,1}\ee_1+q_{j,2}\ee_2+q_{j,3}\ee_{n+1}=(q_{j,1},q_{j,2},0,\ldots,0,q_{j,3})\in\Z^{n+1}, \qquad j\ge j_0(\ux),
	\] 
	is the tail of the sequence of best approximations.
\end{theorem}

In fact, the logarithmic factor in \eqref{eq:mc} is only 
required when $n=3$.
Since the norms of best approximations grow exponentially~\cite{blau},
we see that the condition \eqref{eq:mc} holds as soon as
for some $\epsilon>0$ we have
\[
\vert q_{j,1}\xi_1+q_{j,2}\xi_2+q_{j,3} | < \Vert \qq_{j+1}\Vert^{-n-\epsilon}, \qquad j\ge j_0.
\]
It is not hard to see that this is in turn satisfied if
\begin{equation}  \label{eq:U}
\wo(\xi_1, \xi_2)>n.
\end{equation}
Note that any $\ux\in\R^n$ arising from Theorem~\ref{moschus} is automatically totally irrational as soon as $\xi_1, \xi_2$ 
has this property (otherwise the sequence of best approximations for $\ux$ would terminate).

\subsection{Metric results for Cartesian products and fibers}

Part (I)
of the following is partial claim of~\cite[Proposition~2.3]{mattila},
originally due to Marstrand~\cite{mars} when $n_1=n_2=1$, see Federer~\cite[\S~2.10.25]{fed} for arbitrary dimension. Part (II) can
be obtained by slightly generalizing the proof
of part (c) in the proof of Tricot's~\cite[Theorem~3]{tricot}.

\begin{lemma}[Marstrand; Tricot]  \label{Mars}
	Let $n_1, n_2$ be positive integers and $M\subseteq \R^{n_1+n_2}$
	be measurable. Denote fibers by $M_{\xx}=\{ \yy\in\R^{n_2}: (\xx,\yy)\in M\}$.
	\begin{itemize}
		\item[(I)] If $\dim_H(\{ \xx\in\R^{n_1}: \dim_H(M_{\xx})\ge t\})\ge s$, then $\dim_H(M)\ge s+t$. 
		\item[(II)] If $\dim_P(\{ \xx\in\R^{n_1}: \dim_H(M_{\xx})\ge t\})\ge s$, then $\dim_P(M)\ge s+t$. 
	\end{itemize} 
\end{lemma}

For $A\subseteq \R^{n_1},B\subseteq \R^{n_2}$ any 
non-empty measurable sets,
Lemma~\ref{Mars} implies
\begin{equation}  \label{eq:hdcp}
\dim_H(A\times B) \ge \dim_H(A)+ \dim_H(B),\quad \dim_P(A\times B) \ge \dim_H(A) + \dim_P(B).
\end{equation}
However, we will require the more general properties
of Lemma~\ref{Mars} in place of \eqref{eq:hdcp}.
Conversely, the upper bounds
\begin{equation} \label{eq:yz}
\dim_H(A\times B)\le \dim_H(A)+\dim_P(B)\le  \dim_H(A)+ n_2,
\end{equation}
and
\begin{equation} \label{eq:ttre}
\dim_P(A\times B)\leq \dim_P(A) + \dim_P(B)\le \dim_P(A)+ n_2,
\end{equation}
hold, where the non-obvious left estimates are again part of~\cite[Theorem~3]{tricot}.
The full claim of~\cite[Theorem~3]{tricot} summarizes
the properties \eqref{eq:hdcp}, \eqref{eq:yz}, \eqref{eq:ttre} 
in short as
\begin{align*}
\dim_H(A)+ \dim_H(B)&\le \dim_H(A\times B)\le \dim_H(A) + \dim_P(B)\\ &\le \dim_P(A\times B)\leq \dim_P(A) + \dim_P(B).
\end{align*}


\section{Proof of Theorem~\ref{G}  }  \label{s3}

Let us immediately introduce a subset of $\Gamma_n$
of relevance below. 

\begin{definition}
	Let $\widetilde{\Gamma}_n\subseteq \R^n$ be the set of $\ux\in \Gamma_n$ for which the three-dimensional lattice
	from the definition of $R(\ux)$ can be chosen $\mathcal{L}=\mathcal{L}_{n,3}=\scp{\ee_1,\ee_2,\ee_{n+1}}_{\Z}$.
\end{definition}

It is obvious that $\Gamma_2=\widetilde{\Gamma}_2$ which is just 
the set of totally irrational $\ux\in\R^2$, as well as
\begin{equation} \label{eq:OOO}
\Theta^n\subseteq\widetilde{\Gamma}_n\subseteq \Gamma_n, \qquad n\ge 2.
\end{equation}

\subsection{Proof of lower bounds}  \label{boup}

A key observation is that the work of Das, Fishman, Simmons, Urba\'nski~\cite[\S~3.3]{dfsu2} implies that
\begin{equation} \label{eq:inv}
\mathcal{V}(w):= \dim_H(\{ (\xi_1,\xi_2): \wo(\xi_1,\xi_2)>w\})= 
\frac{2}{w}, \qquad w\in [2+\sqrt{2},\infty].
\end{equation}
Here we implicitly restrict to $(\xi_1,\xi_2)$ totally irrational.
They further provided a different, more complicated, explicit formula for $w<2+\sqrt{2}$ as well that we want to avoid stating.
Formula~\eqref{eq:inv}, but not the formula for $n=3$, follows alternatively from the independent paper~\cite{bcc} and Jarn\'ik's identity
that relates one linear form with simultaneous approximation
for two variables. 

On the other hand, by the observations in~\S~\ref{MO}, 
via using Theorem~\ref{moschus}, we may apply
(I) of Lemma~\ref{Mars} with $M=\widetilde{\Gamma}_n$, and parameters $n_1=2, t=n_2=n-2$ and $s=\mathcal{V}(n)$.
Note hereby that any arising $\ux$ is indeed totally
irrational by the concluding remark in~\S~\ref{MO}.
By \eqref{eq:OOO} and Lemma~\ref{Mars} we infer the lower bound
\[
\dim_H(\Gamma_n)\ge \dim_H(\widetilde{\Gamma}_n)\ge s+t=\mathcal{V}(n)+n-2.
\]
By using \eqref{eq:inv} if $n\ge 4>2+\sqrt{2}$ and after some simplifications of the according formula
when $n=3<2+\sqrt{2}$, the right hand side becomes the respective values in \eqref{eq:mm}.

%

Regarding packing dimension, it follows directly from~\cite[Theorem~3.10]{dfsu2} that the packing dimension of the set involved in \eqref{eq:inv}
is at least $1$ for any $w\ge 2$, with equality for $w\ge 3$.
Combined with (II) of Lemma~\ref{Mars} for the same parameters
$n_i, t$ as above, indeed
\begin{align*}
\dim_P(\Gamma_n) \geq
\dim_P(\widetilde{\Gamma}_n) \geq n-2 + \dim_P(\{ (\xi_1,\xi_2)\in\R^2:\; \wo(\xi_1,\xi_2)>n \})=n-1,
\end{align*}
for $n\ge 3$. For $n=2$ the claim is obvious.


\begin{remark}
	An alternative proof of the bound for the packing dimension follows from \eqref{eq:OOO} and the stronger claim
	$\dim_P(\Theta^n)\ge n-1$ proved 
	in~\S~\ref{sp} below.
\end{remark}

\subsection{Upper bounds: Special three-dimensional lattice} \label{31}

We first show the upper bounds for the smaller set $\widetilde{\Gamma}_n$
where the three-dimensional lattice containing all large best approximations is just $\scp{\ee_1,\ee_2,\ee_{n+1}}_{\Z}$.
From the definition of $\widetilde{\Gamma}_n$
and by Dirichlet's Theorem \eqref{eq:diri}, 
we see that any $\ux=(\xi_1,\ldots,\xi_n)\in \widetilde{\Gamma}_n$ 
satisfies
\[
\wo(\xi_1, \xi_2)\ge n.
\]
Thus
\begin{equation}  \label{eq:geduld}
\widetilde{\Gamma}_n\subseteq \{ (\xi_1,\xi_2)\in\R^2: \wo(\xi_1, \xi_2)\ge n  \} \times \R^{n-2}, \qquad n\ge 2.
\end{equation}
Combined with \eqref{eq:yz}, we get
\[
\dim_H(\widetilde{\Gamma}_n)\le n-2+ \dim_H(\{ (\xi_1,\xi_2)\in\R^2: \wo(\xi_1, \xi_2)\ge n  \}).
\]
As previously noticed,
by~\cite[Theorem~4.9]{dfsu2} the right dimension is again $\mathcal{V}(n)$ as for the sets in \eqref{eq:inv}
where strict inequality is imposed. This proves the reverse upper bound
for the Hausdorff dimension of the sets $\widetilde{\Gamma}_n$. 

Combining \eqref{eq:geduld} with \eqref{eq:ttre}, we get that
\[
\dim_P(\widetilde{\Gamma}_n)\le n-2+ \dim_P(\{ (\xi_1,\xi_2)\in\R^2: \wo(\xi_1, \xi_2)\ge n  \})= n-1, \qquad n\ge 3,
\]
where the last identity is again due to~\cite[\S~3.3]{dfsu2}. 
The reverse lower bound $n-1$ for $\dim_P(\widetilde{\Gamma}_n)$ 
(thus also for $\dim_P(\Gamma_n)$) for $n\ge 2$
was already shown in~\S~\ref{boup}, hence identity \eqref{eq:grim} is proved for the smaller sets $\widetilde{\Gamma}_n$.

\subsection{Upper bounds: General case} \label{genc}
We
settle the upper bounds for the larger sets $\Gamma_n$
where the three-dimensional integer lattice is arbitrary.
The main idea is to apply rational automorphisms
of $\R^{n+1}$ to reduce it to the special case of~\S~\ref{31}.
Our proof below performing this in detail is reasonably lengthy 
and may not be the easiest available.

First notice that
since there are only countably many three-dimensional sublattices of $\Z^{n+1}$ and by sigma-additivity of measures, 
it suffices to show that for any 
fixed three-dimensional sublattice $\mathcal{L}$ of $\Z^{n+1}$, the set of $\ux\in\R^n$ inducing some tail of best approximations in $\mathcal{L}$, has Hausdorff and packing dimension at most as in Theorem~\ref{G}. 
Denote by $\Gamma_n(\mathcal{L})\subseteq \Gamma_n\subseteq \R^n$ this set for 
any fixed given sublattice $\mathcal{L}$ of $\Z^{n+1}$. Then $\widetilde{\Gamma}_n= \Gamma_n(\mathcal{L})$
if $\mathcal{L}=\scp{\ee_1,\ee_2,\ee_{n+1}}_{\Z}$.

There is a bijective linear map $f_{\mathcal{L}}: \R^{n+1}\to \R^{n+1}$ induced by
some integer matrix $A_{\mathcal{L}}\in\mathbb{Z}^{(n+1)\times (n+1)}$
that maps $\mathcal{L}$ to the particular sublattice $\scp{\ee_1,\ee_2,\ee_{n+1}}_{\Z}$, as in $\widetilde{\Gamma}_n$.
To see this, we extend a $\Z$-basis of $\mathcal{L}$ 
to any vector basis 
of $\R^{n+1}$ consisting of integer vectors, then map the 
three $\Z$-base vectors of $\mathcal{L}$
to $\ee_1, \ee_2, \ee_{n+1}$ respectively, then extend 
it to an automorphism of $\R^{n+1}$ by mapping the 
remaining $n-2$ integer
base vectors to $\ee_3,\ldots,\ee_n$, and finally 
multiply the arising rational
matrix by the common denominator.
Denote by $g_{\mathcal{L}}: \R^{n+1}\to \R^{n+1}$ the adjoint map
of the inverse $f_{\mathcal{L}}^{-1}$ of $f_{\mathcal{L}}$. Since  $g_{\mathcal{L}}$ is an automorphism as well, 
it is bi-Lipschitz and thus preserves Hausdorff and
packing dimension~\cite{falconer}. Hence if we write
\[
\bg:= \iota(\Gamma_n(\mathcal{L}))
\]
for the isometric image of $\Gamma_n(\mathcal{L})\subseteq \R^n$
into $\R^{n+1}$ via the 
embedding
\[
\iota:\R^n\to \R^{n+1},\qquad (x_1,\ldots,x_n)\to (x_1,\ldots,x_n,1),
\] 
that just equals the $\ux\to \uxs$ map, we have 
\begin{equation} \label{eq:togol}
\dim_H( g_{\mathcal{L}}(\bg)) = \dim_H(\Gamma_n(\mathcal{L})), \quad \dim_P( g_{\mathcal{L}}(\bg)) = \dim_P(\Gamma_n(\mathcal{L})).
\end{equation}
We will bound the dimensions for the left hand side image sets.

Define an affine and a linear hyperplane of $\R^{n+1}$, parallel
to each other, by
\[
\mathcal{A}=\{ (x_1,\ldots,x_{n+1})\in\R^{n+1}:\; x_{n+1}=1 \},\quad
\mathcal{B}=\{ (x_1,\ldots,x_{n+1})\in\R^{n+1}:\; x_{n+1}=0 \}.
\] 
Define further a map $\Delta: \R^{n+1}\to \R^{n+1}$ by
\[
\Delta(\xx)=\begin{cases}
\frac{1}{x_{n+1}}\cdot (x_1,\ldots,x_{n},x_{n+1}),\qquad \text{if}\; \xx\in \R^{n+1}\setminus \mathcal{B}, \\
(x_1,\ldots,x_{n},0), \qquad\qquad\qquad \text{if}\; \xx\in \mathcal{B}.
\end{cases}
\]
Note that $\uxs=\iota(\ux)\in \mathcal{A}$ for any $\ux\in \R^n$
and $\Delta$ is just the identity on $\mathcal{B}$.
Obviously $\Delta$ maps $\R^{n+1}\setminus \mathcal{B}$
onto $\mathcal{A}$. We claim that
when restricting its domain to 
$g_{\mathcal{L}}(\mathcal{A})\setminus \mathcal{B}\supseteq g_{\mathcal{L}}(\bg)\setminus \mathcal{B}$, it is injective. Indeed,
clearly $g_{\mathcal{L}}(\mathcal{A})$ is an affine but not a linear subspace (as the image of an affine, 
non linear subspace under an
automorphism). Thus it has only a singleton as intersection with any line through the origin, proving the claim in view of 
the definition of $\Delta$.
Hence, as $\bg\subseteq \mathcal{A}$ 
and thus $g_{\mathcal{L}}(\bg)\subseteq g_{\mathcal{L}}(\mathcal{A})$, and as $\Delta$ is locally bi-Lipschitz on $\R^{n+1}\setminus \mathcal{B}$, writing $g_{\mathcal{L}}(\bg)\setminus \mathcal{B}$
as a countable union of sets with last coordinate bounded
away from $0$ in absolute value, by an easy
sigma-additivity argument for measures, we have
\[
\dim_H(\Delta( g_{\mathcal{L}}(\bg)\setminus \mathcal{B}))= \dim_H(  g_{\mathcal{L}}(\bg)\setminus \mathcal{B})
\]
and
\[
\dim_P(\Delta( g_{\mathcal{L}}(\bg)\setminus \mathcal{B}))= \dim_P(  g_{\mathcal{L}}(\bg)\setminus \mathcal{B}).
\]
On the other hand, on $\mathcal{B}$ the map $\Delta$ is just the identity. Thus together with \eqref{eq:togol} we easily conclude that
\begin{equation*}
\dim_H(\Delta( g_{\mathcal{L}}(\bg)))= \dim_H(\Gamma_n(\mathcal{L})),\qquad \dim_P(\Delta( g_{\mathcal{L}}(\bg)))= \dim_P(\Gamma_n(\mathcal{L})).
\end{equation*}
Obviously the same identities hold when restricting 
the left hand side sets,
containing only vectors with last coordinate either $0$ or $1$,
to the first $n$ coordinates (i.e. chopping off the 
last coordinate). In other words, if we let 
\[
\pi:\R^{n+1}\to \R^n, \qquad (x_1,\ldots,x_{n+1})\to (x_1,\ldots,x_n),
\]
the projection that reverses $\iota$, denoting this projected set by 
\[
\mathcal{U}_{\mathcal{L}}:= \pi(\Delta( g_{\mathcal{L}}(\bg)))\subseteq \R^n,
\] 
we have
\begin{equation}  \label{eq:ZWE}
\dim_H(\mathcal{U}_{\mathcal{L}})= \dim_H(\Gamma_n(\mathcal{L})),\qquad \dim_P(\mathcal{U}_{\mathcal{L}})= \dim_P(\Gamma_n(\mathcal{L})).
\end{equation}
Therefore it suffices to bound from above the Hausdorff and packing dimensions of $\mathcal{U}_{\mathcal{L}}$ as in Theorem~\ref{G}.

Let $\ux\in \Gamma_n(\mathcal{L})$ so that $\uxs\in \bg$ be arbitrary,
and $\qq\in \mathcal{L}\subseteq \Z^{n+1}$.
Then by definition of $g_{\mathcal{L}}$ we have
\begin{equation} \label{eq:teilchen}
| \uxs \cdot \qq| = | g_{\mathcal{L}}(\uxs)\cdot f_{\mathcal{L}}(\qq)|, 
\end{equation}
and by construction
\begin{equation}  \label{eq:ffF}
f_{\mathcal{L}}(\qq)\in \scp{\ee_1,\ee_2,\ee_{n+1}}_{\Z},
\end{equation}
in particular it is an integer vector. Moreover, as bijective linear
map $f_{\mathcal{L}}$ is bi-Lipschitz so that
\begin{equation} \label{eq:001}
\Vert f_{\mathcal{L}}(\qq) \Vert \asymp \Vert \qq\Vert,
\end{equation}
with some absolute implied constants.
Write $g_{\mathcal{L}}(\uxs)=(\zeta_1,\ldots,\zeta_{n+1})$
and let
\[
\uz= \pi( \Delta( g_{\mathcal{L}}(\uxs) ))=
\begin{cases}
(\zeta_1/\zeta_{n+1},\ldots,\zeta_n/\zeta_{n+1}), \quad\; \text{if} \;\; g_{\mathcal{L}}(\uxs)\notin \mathcal{B},\\
(\zeta_1, \ldots,\zeta_n), \qquad\qquad\qquad \text{if} \;\; g_{\mathcal{L}}(\uxs)\in \mathcal{B}.
\end{cases}
\]
By \eqref{eq:teilchen} we have 
\[
|\Delta(g_{\mathcal{L}}(\uxs))\cdot f_{\mathcal{L}}(\qq)|=
\frac{1}{|\zeta_{n+1}|}\cdot | g_{\mathcal{L}}(\uxs)\cdot f_{\mathcal{L}}(\qq)|= \frac{1}{|\zeta_{n+1}|}\cdot | \uxs\cdot \qq|,
\qquad \text{if} \; g_{\mathcal{L}}(\uxs)\notin \mathcal{B},
\]
and
\[
|\Delta(g_{\mathcal{L}}(\uxs))\cdot f_{\mathcal{L}}(\qq)|=
| g_{\mathcal{L}}(\uxs)\cdot f_{\mathcal{L}}(\qq)|=  | \uxs\cdot \qq|,
\qquad \text{if} \; g_{\mathcal{L}}(\uxs)\in \mathcal{B}.
\]
In any case,
we get that
\begin{equation} \label{eq:002}
|\Delta(g_{\mathcal{L}}(\uxs))\cdot f_{\mathcal{L}}(\qq)|\asymp  | \uxs\cdot \qq|
\end{equation}
where the implied constant is absolute on sets where $\zeta_{n+1}=0$
or $|\zeta_{n+1}|\in (v^{-1},v)$ for any given $v>1$. 

Combining \eqref{eq:001}, \eqref{eq:002} and as we may choose 
$\qq$ best approximations for $\ux$ and by definition
of $\uz$, we get that
\[
\wo(\uz)\ge \wo(\ux)\ge n.
\]
By the special form \eqref{eq:ffF} of the integer vectors $f_{\mathcal{L}}(\qq)$,
it is further clear that the projection of $\uz$ to the first two coordinates has the same property, i.e. 
\begin{equation} \label{eq:YZ}
\wo(\zeta_1/\zeta_{n+1}, \zeta_2/\zeta_{n+1})\ge n, \; (\uz\notin\mathcal{B})\qquad
\wo(\zeta_1, \zeta_2)\ge n,\; (\uz\in\mathcal{B}).
\end{equation}
For $v>1$, define parametric subsets of $\mathcal{U}_{\mathcal{L}}$ given as
\[
\mathcal{U}_{\mathcal{L}}(v):= \pi(\Delta(X_v)), \qquad 
X_v:= \{ \ux\in\R^n: |\zeta_{n+1}|\in \{0\}\cup (v^{-1},v) \}
\]
where $\zeta_{n+1}$ is the last coordinate of $g_{\mathcal{L}}(\uxs)$ as above. Then $\Delta$ is bi-Lipschitz on any $X_v$. Thus
again by the invariance of Hausdorff and
packing dimension under bi-Lipschitz maps, for any $v>1$,
the quantities $\dim_H(\mathcal{U}_{\mathcal{L}}(v))$ and 
$\dim_P(\mathcal{U}_{\mathcal{L}}(v))$ can be estimated
as in~\S~\ref{31} by precisely the same argument via
\eqref{eq:YZ} and \eqref{eq:yz}, \eqref{eq:ttre}. Since we may write
\[
\mathcal{U}_{\mathcal{L}}= \bigcup_{v>1, v\in\Z} \mathcal{U}_{\mathcal{L}}(v)
\]
as a countable union of such sets, again by sigma-additivity of measures the same estimates hold for $\mathcal{U}_{\mathcal{L}}$
and finally in view of \eqref{eq:ZWE} for $\Gamma_n(\mathscr{L})$ as well.
%
%
%
%

\begin{remark}
	We cannot conclude that $f_{\mathcal{L}}(\qq)\in\Z^{n+1}$ are best approximations for $\uz=\pi(\Delta(g_{\mathcal{L} }((\uxs)))\in\R^n$. However, it suffices for the argument that they induce approximations of order at least $n$.
\end{remark}

\section{Sketch of the Proof of Theorem~\ref{N} }

By a slightly more general version of Theorem~\ref{moschus}
from~\cite{ngm}, again for any element of
$\{(\xi_1,\ldots,\xi_{k-1})\in\R^{k-1}: \wo(\xi_1,\ldots,\xi_{k-1})>n  \}$, we get some full measure set
$F_{n,k}=F_{n,k}(\xi_1,\ldots,\xi_{k-1})\subseteq \R^{n-k+1}$
so that $\ux=(\xi_1,\ldots,\xi_n)$ lies in $\widetilde{\Gamma}_{n,k}$
for any $(\xi_k,\ldots,\xi_{n})\in F_{n,k}$.
Here $\widetilde{\Gamma}_{n,k}\subseteq \Gamma_{n,k}$ is defined likewise as $\widetilde{\Gamma}_n=\widetilde{\Gamma}_{n,3}$ from 
\S~\ref{s3}
with respect to the $k$-dimensional lattice $\mathcal{L}_{n,k}:= \scp{\ee_1,\ldots,\ee_{k-1},\ee_{n+1}}_{\Z}$. 
Conversely, very similarly as in~\S~\ref{31} we get
\begin{equation*} 
\widetilde{\Gamma}_{n,k}\subseteq \{ \ux\in\R^{k-1}: \wo(\ux)\ge n  \} \times \R^{n-k+1}.
\end{equation*}
Combining these properties 
with (I) of Lemma~\ref{Mars} and \eqref{eq:yz} yield the claims 
\eqref{eq:RHs}, \eqref{eq:Rhs}
on the Hausdorff dimension for the smaller sets $\widetilde{\Gamma}_{n,k}$.
Very similarly as in~\S~\ref{genc}, via rational automorphisms that map a given $k$-dimensional rational lattice $\mathcal{L}\subseteq \Z^{n+1}$ to $\mathcal{L}_{n,k}$,
we lift the upper bound to the larger set $\Gamma_{n,k}$.

Regarding packing dimension, we have
\[
\dim_P(\{ \ux\in\R^{k-1}: \wo(\ux)>n  \})= \dim_P(\{ \ux\in\R^{k-1}: \wo(\ux)\ge n  \})=k-2, \qquad 3\le k\le n,
\]
as can be seen via~\cite[Theorems~3.8 \& 4.9]{dfsu2} with a short calculation.
Thus, using the above observations on $\widetilde{\Gamma}_n$,
 we conclude with part (II) of Lemma~\ref{Mars} where $n_1=k-1, n_2=t=n-k+1$ (for lower bounds) and \eqref{eq:ttre} (for upper bounds) that 
\begin{equation} \label{eq:KKo}
\dim_P(\widetilde{\Gamma}_{n,k}) = (n-k+1)+(k-2)=n-1.
\end{equation}
Finally again
similarly as in~\S~\ref{genc} we can lift the upper bound to the sets
$\Gamma_{n,k}$, the reverse inequality being
a trivial consequence of \eqref{eq:KKo}, hence \eqref{eq:rhS} holds. 

\section{Proof of Theorem~\ref{Y} }

The lower bounds are clear by Theorem~\ref{H}, we need to prove the upper estimates.
Since there are only countably many finite subsets of 
sublattices of $\Z^{n+1}$ and 
by sigma-additivity of measures,
there is a subset $\mathcal{Z}^{n,k}\subseteq \mathcal{Y}^{n,k}$ 
with the property that 
$\dim_H(\mathcal{Y}^{n,k})=\dim_H(\mathcal{Z}^{n,k})$ and so that
the finite collections of $k$-dimensional sublattices $\mathcal{L}_1,
\ldots,\mathcal{L}_{i(\ux)}$ of $\Z^{n+1}$ are the same for any $\ux\in \mathcal{Z}^{n,k}$. So assume this set
of lattices $\mathcal{L}_1,\ldots, \mathcal{L}_i$ is fixed.
Clearly by pigeon hole principle for any $\ux\in \mathcal{Z}_n$ 
there is some lattice $\mathcal{L}_j$, $j=j(\ux)\in\{ 1,2,\ldots,i\}$
containing infinitely many best approximations.
Again by additivity of measures, 
it suffices to treat the case where this lattice is the same
for any $\ux\in \mathcal{Z}^{n,k}$. Without loss of generality 
we can assume it is $\mathcal{L}:=\mathcal{L}_1$. However, then 
within $\mathcal{L}$ we find infinitely many vectors inducing approximations of order $>\wo(\ux)-\varepsilon$. If 
$\mathcal{L}=\mathcal{H}_1=\scp{\ee_1,\ldots,\ee_{k-1}, \ee_{n+1}}_{\Z}$, then it follows from Dirichlet's Theorem \eqref{eq:diri}
that $\om(\xi_1,\ldots,\xi_{k-1})\ge \wo(\ux)\ge n$ for any $\ux=(\xi_1,\ldots,\xi_n)\in \mathcal{Z}^{n,k}$. 
Otherwise, we extend any rational linear map sending
any base of $\mathcal{L}$ bijectively to $\ee_1,\ldots,\ee_{k-1}, \ee_{n+1}$ to a rational automorphism of $\R^{n+1}$ and argue very similarly as in~\S~\ref{genc} to get a set of the same Hausdorff dimension
$\dim_H(\mathcal{Z}^{n,k})$ where this is the case. 
Hence \eqref{eq:yz} and a formula generalising a classical result of
Jarn\'ik~\cite{jarnik} to higher dimension, 
see for example~\cite{bdv}, imply
\[
\dim_H(\mathcal{Y}^{n,k})= \dim_H(\mathcal{Z}^{n,k})\le n-k+1+ \dim_H(\{ \ux\in\R^{k-1}: \om(\ux)\ge n \})= n-1+\frac{k}{n+1}.
\]

\section{ Proof of Theorem~\ref{H}: Upper bounds} \label{sieben}

The upper bounds in \eqref{eq:uc} and
\eqref{eq:DREI} follow immediately from Theorem~\ref{G} 
and \eqref{eq:absaty}. We are left with the proof of \eqref{eq:zzwei}.
For this we use a different strategy.
We show that for any $\ux=(\xi_1,\xi_2)\in \Theta^2$ we have
\begin{equation} \label{eq:fasto}
\min\{ \om(\xi_1), \om(\xi_2) \} \ge 4.
\end{equation}
If this is true then a classical metrical formula by Jarn\'ik~\cite{jarnik} and \eqref{eq:yz} indeed imply
\begin{equation} \label{eq:himp}
\dim_H(\Theta^2) \le 1 + \dim_H(\{ \xi\in\R: \om(\xi)\ge 4 \})=
1+\frac{2}{4+1}=\frac{7}{5}.
\end{equation}
Let $\qq=(q_1, q_2, q_3)\in\Z^3$ be a best approximation of large norm
for $\ux$. Without loss of generality we can 
assume $\qq\in \mathcal{H}_2$, thus $q_1=0$. 
Let $\mu$ be implictly defined by
\begin{equation} \label{eq:Cf}
|\qq\cdot \uxs|= |q_2\xi_2 + q_3|= \Vert \hat\qq\Vert^{-\mu}. 
\end{equation}
Then 
by Dirichlet's Theorem \eqref{eq:diri} and since 
$\qq$ is a best approximation, we have $\mu\ge 2$.
Then $-q_3/q_2$ is a convergent to $\xi_2$.
Note further that by \eqref{eq:Cf} and the theory of continued fractions, the next convergent has denominator at least $\Vert \hat\qq\Vert^{\mu}/2$ (see~\cite[Proposition~5.2]{JS}). 
Hence, any integer 
linear form $a\xi_2+b$ with $\max\{ |a|, |b| \}< \Vert \hat\qq\Vert^{\mu}/2$ 
satisfies $|a\xi_2+b|\ge \Vert \hat\qq\Vert^{-\mu}$. In other
words, there is no better approximation for $\ux$ within $\mathcal{H}_2$
up to norm $\Vert \hat\qq\Vert^{\mu}/2$.
On the other hand, 
by Dirichlet's Theorem 
there is some
$\pp=(p_1,p_2,p_3)\in\Z^{3}$ inducing 
\begin{equation} \label{eq:Qere}
0<\Vert \hat\pp\Vert\le 2\Vert \hat\qq\Vert^{\mu/2}, \qquad |\pp\cdot \uxs|\le \Vert \hat\pp\Vert^{-2}\le \frac{1}{4}\Vert \hat\qq\Vert^{-\mu} < \Vert \hat\qq\Vert^{-\mu}.
\end{equation}
By \eqref{eq:Cf} clearly $\pp\ne \qq$. We may assume $\pp$ is a best
approximation for $\ux$, so since $(\xi_1,\xi_2)\in\Theta^2$ and the above argument excludes $\pp\in \mathcal{H}_2$, we must have
$\pp\in \mathcal{H}_1$. Clearly $\Vert \hat\pp\Vert> \Vert \hat\qq\Vert$. 
Let $\rr\in\Z^3$ be the best approximation for $\ux$ following $\pp$.
By a very similar argument as above based on Dirichlet's Theorem 
for $\pp$ in place of $\qq$, 
we can now exclude
$\rr\in \mathcal{H}_1$, so we must have $\rr\in \mathcal{H}_2$.
But then $\Vert \rr\Vert\ge \Vert \hat\qq\Vert^{\mu}/2$ by the above
observation. Hence
there is no other best approximation between $\Vert \hat\pp\Vert$ 
and $\Vert \hat\qq\Vert^{\mu}/2$, 
so $\pp$ minimizes $|\uu\cdot \uxs|$
among all integer 
vectors $\uu\in\Z^3$ with $\Vert \hat\uu\Vert< \Vert \hat\qq\Vert^{\mu}/2$. On the other hand, again by Dirichlet's Theorem
\begin{equation} \label{eq:VZ}
0<\Vert \hat\vv\Vert \le \Vert \hat\qq\Vert^{\mu}/3, \qquad
|\vv\cdot \uxs|\le \Vert \hat\vv\Vert^{-2}\le \frac{1}{9}\Vert \hat\qq\Vert^{-2\mu} 
\end{equation}
has a solution $\vv\in\Z^3$. 
Again we can assume $\vv$ is a best approximation, hence
the above observation that there is no best approximation
with norm in $(\Vert \hat\pp\Vert,\Vert \hat\qq\Vert^{\mu}/2)$ implies $\vv=\pp$. Combining the right estimate
from \eqref{eq:VZ} with the left bound from 
\eqref{eq:Qere} and this happens for infinitely many
$\pp\in \mathcal{H}_1$ as above,
we see that $\om(\xi_1)\ge 4$. An analogous argument 
yields $\om(\xi_2)\ge 4$ as well, hence
\eqref{eq:fasto} is proved.

\begin{remark} \label{RR}
	The argument in fact shows that the best approximations
	in $\mathcal{H}_1$ and $\mathcal{H}_2$ must occur at some 
	high rate when \eqref{eq:fasto} is close to optimal. Using the variational principle~\cite{dfsu1, dfsu2}, with some effort
	some stronger bound in the interval $(1,7/5)$ can be obtained, however we omit its slightly techincal explicit calculation.
\end{remark}

\begin{remark}
	An analogous argument shows in general that $\ux=(\xi_1,\ldots,\xi_n)\in \Theta^n$
	implies
	\[
	\min\{ \om(\xi_1), \om(\xi_2) \} \ge n^2.
	\]
	We can conclude
	\[
	\Theta^n\subseteq \R^{n-2} \times \{ \xi\in\R: \om(\xi)\ge n^2 \}^2,
	\]
	and by \eqref{eq:yz} go on to estimate
	\begin{equation} \label{eq:TT}
	\dim_H(\Theta^n)\le n - 2+ \dim_H(\{ \xi\in\R: \om(\xi)\ge n^2 \}^2).
	\end{equation}
	However, by~\cite[Theorem~3.3]{JS}, for $n\ge 3$ the right hand side in \eqref{eq:TT} is at least $n-1$, in particular the right expression exceeds twice the single dimension
	of its factors. 
	Hence it seems the bound in \eqref{eq:uc}
	cannot be reached with this method. This argument 
	is most likely true for $n=2$ as well (see~\cite[Conjecture~3]{JS}), so just \eqref{eq:fasto}
	may be insufficient to improve on \eqref{eq:himp} either and
	the bound $1$ seems to be the optimal outcome of the method.
\end{remark}


\section{Proof of Theorem~\ref{H}: Lower bounds} 

\subsection{Outline}

We first show in~\S~\ref{s1} 
that for $n=2$, there exist vectors $(\xi_1,\xi_2)\in\R^2$
with properties (i)-(v), and with a slight twist of (vi).
The transition to general $n$ as well
as the weaker lower bound $n-2$ for the Hausdorff dimension in~\S~\ref{lbb} will then be an easy consequence of Theorem~\ref{moschus} above obtained in~\cite{ngm}. By modifications of the method, the stronger metrical claims will be proved in~\S~\ref{s4}, \ref{sp}. 

\subsection{Existence claim: Case $n=2$}   \label{s1}

In this section, we prove.

\begin{theorem}  \label{thm1}
	There exist uncountably many totally irrational $\ux\in \R^2$ for which
	the tail of the best approximation sequence with respect to the maximum norm lies in the union
	of the two 2-dimensional sublattices of $\Z^{3}$ given by
	\[
	\mathcal{H}_1= \{ (x,y,z)\in \Z^3: y=0 \}, \qquad \mathcal{H}_2= \{ (x,y,z)\in \Z^3: x=0 \}.
	\]
	Moreover, large best approximations alternately 
	lie in $\mathcal{H}_1$
	and $\mathcal{H}_2$. Furthermore, for any given $w\in[2,\infty]$ we can 
	choose $\ux$ so that additionally $\wo(\ux)=w$.
\end{theorem}

Note that the condition $\om(\ux)=w^2$ in (vi) is missing 
for a full analogue of Theorem~\ref{H}. Indeed, the vectors $\ux$
constructed in this section satisfy $\om(\ux)=w^2-1$ instead.

We construct our real vector. 
Let $\tau>1+\sqrt{2}$ be a parameter.
Let $\alpha_j, \beta_j$ be increasing positive integer sequences
and derive the integers
\[
A_j= 2^{\alpha_j}, \qquad B_j= 3^{\beta_j },
\]
satisfying
\begin{equation} \label{eq:accto}
A_{j+1}\asymp B_j^{\tau}, \qquad B_j\asymp A_j^{\tau}.
\end{equation}
Clearly such choices are possible. Then in particular
\begin{equation}  \label{eq:true}
A_{j+1}\asymp A_j^{\tau^2}, \qquad B_{j+1}\asymp B_j^{\tau^2}.
\end{equation}
Then, upon changing initial terms if necessary, we can assume
\begin{equation} \label{eq:vf}
0<\alpha_1 < \beta_1 < \alpha_2 < \beta_2 < \cdots , \qquad 
1<A_1 < B_1 < A_2 < B_2 < \cdots.
\end{equation}
Finally let $\ux=(\xi_1, \xi_2)$ with
\[
\xi_1= \sum_{j\ge 1} A_j^{-1 } , \qquad \xi_2= \sum_{j\ge 1} B_j^{-1 }.
\]
We claim that it satisfies the assertions of the theorem.

Put
\[
F_j= A_j(A_1^{-1}+\cdots+A_j^{-1})\in \mathbb{N}, \qquad G_j= B_j(B_1^{-1}+\cdots+B_j^{-1})\in\mathbb{N}.
\]
Then $F_j\equiv 1\bmod 2$
and $G_j\equiv 1\bmod 3$ imply the coprimality assertions
\begin{equation}  \label{eq:witt}
(A_j, F_j)=(F_j,2)=1, \qquad (B_j,G_j)=(G_j,3)=1.
\end{equation}
Let
\[
\vv_j=(A_j,0,-F_j)\in \mathcal{H}_1, \qquad 
\w_j=(0, B_j,-G_j)\in \mathcal{H}_2.
\]
Then obviously
\begin{equation} \label{eq:gl}
\Vert\hat\vv_j\Vert= A_j, \qquad \Vert\hat\w_j\Vert= B_j.
\end{equation}
Thus by \eqref{eq:vf} clearly
\begin{equation}  \label{eq:either}
\Vert\hat\vv_j\Vert< \Vert\hat\w_j\Vert< \Vert\hat\vv_{j+1}\Vert, \qquad j\ge 1.
\end{equation}
Then by \eqref{eq:true} moreover 
\begin{equation} \label{eq:ti}
|\vv_j\cdot \uxs|= |A_j \cdot \xi_1 +0\cdot \xi_2 - F_j|= A_j(A_{j+1}^{-1}+A_{j+2}^{-1}+\cdots) \asymp A_j A_{j+1}^{-1}\asymp A_j^{-(\tau^2-1) }
\end{equation}
and
\begin{equation} \label{eq:tz}
|\w_j\cdot \uxs|= |0\cdot \xi_1 + B_j \cdot \xi_2 - G_j|= B_j(B_{j+1}^{-1}+B_{j+2}^{-1}+\cdots)\asymp B_jB_{j+1}^{-1}\asymp B_j^{-(\tau^2-1) },
\end{equation}
are small linear form for $j\ge 1$, when $\tau$ is large.
We show that $\vv_j$ and $\w_j$ precisely comprise all large best approximations. This obviously finishes the proof.

Let $\bb$ be any best approximation. Then by \eqref{eq:either} 
there is an index $j$
such that either $\Vert\hat\vv_j\Vert\le \Vert \hat\bb\Vert<\Vert\hat\w_j\Vert$ or $\Vert\hat\w_j\Vert\le \Vert \hat\bb\Vert< \Vert\hat\vv_{j+1}\Vert$.  We show that in the first case $\bb=\pm \vv_j$, and in the latter case $\bb=\pm \w_j$. 
Assume the first case, so
\begin{equation} \label{eq:T}
\Vert \hat\vv_j\Vert\le \Vert \hat\bb\Vert< \Vert \hat\w_j\Vert.
\end{equation}
 the latter works very similarly by symmetry.
First observe that since $\bb$ is a best approximation
of norm at least $\Vert\hat\vv_j\Vert$, we know that
\begin{equation}  \label{eq:see}
|\bb\cdot \uxs|\le |\vv_j\cdot \uxs|.
\end{equation}
We distinguish two cases.

Case 1: $\bb$ lies in the two-dimensonal subspace 
of $\R^3$ spanned by $\vv_j, \w_j$, i.e. 
$\bb\in \scp{\vv_j, \w_j}_{\mathbb{R}}\cap \mathbb{Z}^3$. The special form of $A_j, B_j$ and \eqref{eq:witt} imply the following crucial result on integer
vectors in the two-dimensional lattices $\scp{\vv_j, \w_j}_{\mathbb{R}}\cap \mathbb{Z}^3$.

\begin{proposition}  \label{prop}
	For $\vv_j, \w_j$ as above, if a linear combination
	$g\vv_j+h\w_j$ is an integer vector, then 
	in fact $g\in\Z$ and $h\in \Z$. In other words,
	$\scp{\vv_j, \w_j}_{\mathbb{R}}\cap \mathbb{Z}^3=\scp{\vv_j, \w_j}_{\mathbb{Z}}$.
\end{proposition}

\begin{proof}
Clearly we must have $g,h\in \mathbb{Q}$. 
If we write $(p_1/q_1)\vv_j + (p_2/q_2)\w_j$ with $p_i/q_i$
in lowest terms, then 
it is clear that $q_1$ must be a non-negative integer power of $2$ and $q_2$ a non-negative integer power of $3$ to make the first two coordinates $(p_1/q_1)A_j=(p_1/q_1)2^{\alpha_j}$ resp. $(p_2/q_2)B_j=(p_2/q_2)3^{\beta_j}$ of $g\vv_j+h\w_j$ integers. 
But then by \eqref{eq:witt}
clearly the third coordinate $(p_1/q_1)F_j+(p_2/q_2)G_j$ is not an integer unless
$q_1=q_2=1$.  
\end{proof}

By the proposition applied to $\bb$ and since $\Vert \hat\bb\Vert<\Vert\hat\w_j\Vert$ obviously we must
have $h=0$. Hence
$\bb=g\vv_j$ is an integer multiple of $\vv_j$, but 
since $\bb$ is a best approximation and thus primitive
this integer must be $g=\pm 1$. Hence indeed
$\bb=\pm \vv_j$. The case $\Vert\hat\w_j\Vert\le \Vert \hat\bb\Vert< \Vert\hat\vv_{j+1}\Vert$ works very
similarly by symmetry and yields for the best approximation 
the only candidates $\pm \w_j$.

Case 2: $\bb$ does not lie in the space spanned by $\vv_j, \w_j$.
For this case we use an easy consequence of Minkowski's Second Convex
Body Theorem.

\begin{lemma}  \label{lemur}
	There exists a constant $c>0$ such that for any $\ux\in \R^2$ and
	any parameter $Q\ge 1$, the system
	\[
	|b_1|\le Q,\qquad |b_2|\le Q,\qquad  |b_1 \xi_1 + b_2 \xi_2+b_{3}| < cQ^{-2}
	\]
	does not have three linearly independent solutions in integer
	vectors $\bb=(b_1,b_2,b_3)$.
\end{lemma}

\begin{proof}
Consider the integer lattice $\Z^3$ and the box of $(x_1,x_2,x_3)\in \R^3$ 
with coordinates
\[
|x_1|\le Q, \qquad |x_2|\le Q,\quad  |\xi_1 x_1+\xi_2 x_2 + x_3|\leq cQ^{-2}.
\]
It has volume $8c$, independent of $Q$ and $\xi_1, \xi_2$. Hence, by Minkowski's Second Convex
Body Theorem, the product of the induced successive minima 
is $\ll c$, hence choosing $c$ small enough the third successive minimum is smaller than $1$. This means there
cannot be three linearly independent integer points within the box, which in turn is equivalent to the claim.
\end{proof}

We first notice that both $\vv_j$ and $\w_j$ induce approximations of order greater than two. By \eqref{eq:gl}, \eqref{eq:ti}
and as our choice of $\tau>1+\sqrt{2}$
that implies $(\tau^2-1)/\tau>2$, for some $\epsilon=\epsilon(\tau)>0$
we get
\begin{equation} \label{eq:eins}
|\vv_j\cdot \uxs| \asymp A_j^{-(\tau^2-1)} \asymp B_j^{-(\tau^2-1)/\tau} < B_j^{-2-\epsilon}=
\Vert \hat\w_j\Vert^{-2-\epsilon} 
\end{equation}
and for $\w_j$ by \eqref{eq:tz} we have a stronger estimate that also yields
\begin{equation}  \label{eq:zwei}
|\w_j\cdot \uxs| \asymp B_j^{-(\tau^2-1)} < B_j^{-2-\epsilon} = \Vert \hat\w_j\Vert^{-2-\epsilon}.
\end{equation}
Combined with \eqref{eq:either}, \eqref{eq:T}, \eqref{eq:see}, we have 
\[
\max\{ \Vert \hat\bb\Vert, \Vert \hat\vv_j\Vert, \Vert \hat\w_j\Vert \}= \Vert \hat\w_j\Vert, \qquad \max\{|\bb\cdot \uxs|, |\vv_j\cdot \uxs|,|\w_j\cdot \uxs| \} \ll \Vert \hat\w_j\Vert^{-2-\epsilon}.
\]
By the assumptions of Case 2, the three vectors
$\vv_j, \w_j, \bb$ are linearly independent.
So we get a contradiction
to Lemma~\ref{lemur} for $Q=\Vert \hat\w_j\Vert$, as soon as $\Vert \hat\w_j\Vert$ is sufficiently large. Hence in total Case 2 provides only finitely many best approximations, of small norm.

Combining our observations from Case 1 and Case 2, 
we see that $\vv_j$ and $\w_j$ comprise
all best approximations of large enough norm, as desired. Moreover it is clear
that any $\ux$ as above is totally irrational and by the freedom 
in the choice of $A_i, B_i$ the set of induced $\ux$ is uncountable.
Finally it is easy to check from \eqref{eq:eins}, \eqref{eq:zwei},
the fact that $\vv_j, \w_j$ comprise all the best approximations
and \eqref{eq:gl} that
\begin{equation} \label{eq:wdach}
\wo(\ux)= \frac{\tau^2-1}{\tau}, \qquad \om(\ux)= \tau^2-1.
\end{equation}
So choosing $\tau>1+\sqrt{2}$ appropriately, we can realize any uniform
exponent in $(2,\infty)$. Finally, small modifications of the construction allow for obtaining the endpoints $2$ and $\infty$ as well.

\subsection{General case and lower bound $\dim_H(\Theta^n)\ge n-2$ }  \label{lbb}

As indicated before,
the extension to the general case works with Theorem~\ref{moschus}.
In view of the sufficient condition \eqref{eq:U} and \eqref{eq:wdach}, the lower bound 
\[
\dim_H(\Theta^n) \ge n-2
\] 
follows 
by taking any
$\xi_1, \xi_2$ constructed in~\S~\ref{s1} upon increasing
$\tau$ if necessary, and extending it to $n$-dimensional real
vectors via Theorem~\ref{moschus} (we do not need Lemma~\ref{Mars} here).
As noticed in~\S~\ref{MO} any arising $\ux\in\R^n$ is automatically
totally irrational since $(\xi_1, \xi_2)$ has this property.
In fact the same bound holds 
when restricting to vectors with arbitrary uniform 
exponent $\wo(\ux)=w\in [n,\infty]$, similar as in $\Theta_n(w)$ 
but with some altered value for the ordinary exponent 
$\om(\ux)$ in terms of $w$.

To improve the bound $n-2$, in the next section we 
generalize the construction of~\S~\ref{s1} to obtain some Cantor type
set with the properties of Theorem~\ref{H}, and determine a stronger lower bound for its Hausdorff dimension using results from~\cite{sun, dfsu1, dfsu2}.

\subsection{Proof of \eqref{eq:theclaim}, up to strictness }  \label{s4}

In this section, we show the improved lower bound 
\begin{equation} \label{eq:nonicht}
\dim_H(\Theta_n(w))\ge n-2+\frac{1}{w^2+1}, 
\end{equation}
and thus as $w$ can be arbitrarily close to $n$ also
$\dim_H(\Theta^n)\ge n-2+1/(n^2+1)$. Up to 
the latter inequality not being strict yet, this agrees with claim
\eqref{eq:theclaim} in Theorem~\ref{H}.

For $\tau>n$,
we now consider modified $\xi_1, \xi_2$, with sequences
of continued fraction convergents $(r_{i,j}/s_{i,j})_{j\ge 1}$, $i=1,2$, with the following denominator growth properties:
\begin{equation} \label{eq:YE}
\lim_{j\to\infty} \frac{ \log s_{2,j} }{\log s_{1,j}} = \tau, \qquad  \lim_{j\to\infty} \frac{ \log s_{1,j+1} }{\log s_{2,j}} = \tau.
\end{equation}
Then
\begin{equation}  \label{eq:tirE}
\lim_{j\to\infty} \frac{ \log s_{i,j+1} }{\log s_{i,j}} = \tau^2, \qquad i=1,2.
\end{equation}
We then follow the proof above with 
$A_{i}$ resp. $B_i$ replaced by $s_{1,j}$ resp. $s_{2,j}$,
and $F_{i}$ resp. $G_i$ replaced by $r_{1,j}$ resp. $r_{2,j}$,
and we replace $\vv_j, \w_j$ by
\begin{equation} \label{eq:gjhj}
\g_{j}=s_{1,j}\ee_{1}-r_{1,j}\ee_{n+1}\in\Z^{n+1},\qquad \h_{j}=s_{2,j}\ee_{2}-r_{2,j}\ee_{n+1}\in\Z^{n+1}. 
\end{equation}
In order to establish an analogue of Proposition~\ref{prop}, 
we need coprimality conditions for the denominators, concretely
it suffices to guarantee
\begin{equation}  \label{eq:vieE}
(s_{1,j}, s_{2,j})= 1=(s_{2,j}, s_{1,j+1}), \qquad\quad j\ge 1.
\end{equation}
Write $\mathcal{K}(\tau)$ for the set of $(\xi_1,\xi_2)\in\R^2$ 
satisfying \eqref{eq:YE}, \eqref{eq:tirE}, \eqref{eq:vieE}.

Note that now 
by the theory of continued fractions
\[
|s_{1,j}\xi_1-r_{1,j}| \asymp s_{1,j}^{-\tau^2}, \qquad |s_{2,j}\xi_1-r_{2,j}| \asymp s_{2,j}^{-\tau^2},
\]
slightly stronger than in~\S~\ref{s1} where the approximations
were of order $\tau^2-1$. Together with \eqref{eq:YE} it follows
easily that
$\wo(\xi_1,\xi_2)\ge \tau^2/\tau=\tau$. Thus if $\tau>n$,
we can apply Theorem~\ref{moschus} again to see that for
any $(\xi_1,\xi_2)\in \mathcal{K}(\tau)$ there is
a full measure set $F_n(\xi_1,\xi_2)\subseteq \R^{n-2}$, so for
any $(\xi_3,\ldots,\xi_n)\in F_n(\xi_1,\xi_2)$ the vector
$(\xi_1,\ldots,\xi_n)$ has essentially 
the same best approximations as $(\xi_1, \xi_2)$ (with zeros
added for entries at positions $3,4,\ldots,n$). 
As in Case 1 in~\S~\ref{s1}, via an analogous claim to Proposition~\ref{prop} for $\g_j, \h_j$, 
we can conclude that any 
best approximation within the space spanned by a pair $\g_j, \h_j$ or
a pair $\h_j,\g_{j+1}$ is actually equal to $\g_j$ or $\h_j$.
Moreover, similar to Case 2 in~\S~\ref{s1} by $\wo(\xi_1,\xi_2)\ge \tau>n\ge 2$ via Lemma~\ref{lemur} we see that large best approximations must lie in such spaces,
hence the large best approxmations are precisely the $\g_j,\h_j$.

Thus we have that the $\g_j$ and $\h_j$ again comprise
all best approximations for any $\ux$ as above,
i.e. $(\xi_1, \xi_2)\in\mathcal{K}(\tau)$ and 
$(\xi_3,\ldots,\xi_n)\in F_n(\xi_1,\xi_2)$. 
Moreover it is clear that
\[
|\g_j\cdot \uxs|=|s_{1,j}\xi_1-r_{1,j}| \asymp s_{1,j}^{-\tau^2}= \Vert \hat\g_j\Vert^{-\tau^2}, \quad |\h_j\cdot \uxs|=|s_{2,j}\xi_1-r_{2,j}| \asymp s_{2,j}^{-\tau^2}= \Vert \hat\h_j\Vert^{-\tau^2},
\]
hence $\wo(\ux)=\tau$ and
$\om(\ux)=\tau^2$, so (vi) holds as well. Thus any vector of the form $(\xi_1,\xi_2)\times F_n(\xi_1,\xi_2)$ with $(\xi_1,\xi_2)\in \mathcal{K}(\tau)$ is contained in $\Theta_n(\tau)$. By (I) of Lemma~\ref{Mars} with $M=\Theta_n(\tau)$ and $n_1=2, n_2=t=n-2$ we conclude
\[
\dim_H(\Theta_n(\tau))\ge n-2+ \dim_H(\mathcal{K}(\tau)),\qquad n\ge 2,\;\tau>n.
\]
To finish the proof of \eqref{eq:nonicht},
we identify $w$ with $\tau$ and show 
\begin{equation} \label{eq:novi}
\dim_H(\mathcal{K}(\tau))
\ge \frac{1}{\tau^2+1}, \qquad n\ge 2,\; \tau>n.
\end{equation}
As $\tau$ can be taken arbitrarily close
to $n$, the claimed lower bound for $\Theta^n$ follows as well. 

To show \eqref{eq:novi}, we notice first that by Sun~\cite{sun} (alternatively
this follows easily from the variational principle~\cite{dfsu1, dfsu2}, see also~\cite{tan})
each single set of $\xi_1\in\R$ inducing \eqref{eq:tirE}
for $i=1$ has precisely this Hausdorff dimension $1/(\tau^2+1)$.
Let $\Pi(x,y)=x$ be the projection $\R^2\to \R$ to the first coordinate.
Assume we have already shown that
\begin{equation}  \label{eq:EE}
\Pi(\mathcal{K}(\tau)) \supseteq \{ \xi_1\in\R: \eqref{eq:tirE}\;  holds\; for \; i=1 \}.
\end{equation}
Then,
since the projection of a set has at most the Hausdorff dimension of the original set (by Lipschitz property of projections~\cite{falconer}), the claim \eqref{eq:novi} follows via
\begin{equation} \label{eq:middle}
\dim_H(\mathcal{K}(\tau))\ge \dim_H(\Pi(\mathcal{K}(\tau)))\ge
\dim_H(\{ \xi_1\!\in\!\R: \eqref{eq:tirE}\;  holds\; for \; i=1 \}) = \frac{1}{\tau^2+1}.
\end{equation}
We verify \eqref{eq:EE} to finish the proof. We need
to show that for any $\xi_1$ as above (i.e.
with condition \eqref{eq:tirE} for $i=1$),
there exists $\xi_2\in\R$ as above, that may depend on $\xi_1$, i.e. so that conditions \eqref{eq:YE} and \eqref{eq:vieE} hold as well 
(then \eqref{eq:tirE} holds for $i=2$ as well; in fact there is set identity in \eqref{eq:EE}).
So let arbitrary $\xi_1$ with property \eqref{eq:tirE} for $i=1$
with convergent sequence $(r_{1,j}/s_{1,j})_{j\ge 1}$
be given. Assume we have constructed the partial
quotients of $\xi_2$ up to convergent $r_{2,j-1}/s_{2,j-1}=[a_{2,0};a_{2,1},\ldots,a_{2,j-1}]$ for given $j$, and let $a_{2,j}$ be the next partial quotient for $\xi_2$
to be fixed. 
We may assume $s_{2,j-1}< s_{1,j}$ in view of \eqref{eq:YE}.
From the recursion
\begin{equation} \label{eq:reku}
s_{2,j-2}+a_{2,j}s_{2,j-1}=s_{2,j}
\end{equation} 
for convergent denominators,
there are
\begin{equation}  \label{eq:tima}
\gg \frac{s_{1,j}^{\tau}}{s_{2,j-1}}\ge \frac{s_{1,j}^{\tau}}{s_{1,j}}= s_{1,j}^{\delta}, \qquad \delta=\tau-1>0
\end{equation}
many consecutive partial quotients $a_{2,j}$ that induce $s_{2,j-2}+a_{2,j}s_{2,j-1}=s_{2,j}\asymp s_{1,j}^{\tau}$, as we need for \eqref{eq:YE} in the current step. We must show some of them
induce \eqref{eq:vieE} as well.

On the other hand, by \eqref{eq:tirE} for $i=1$, there are $\ll \log (s_{1,j}s_{1,j+1})\ll \log s_{1,j}$ many primes
dividing either $s_{1,j}$ or $s_{1,j+1}$. Denote $S_j=\{ p\in\mathbb{P}: p|(s_{1,j}s_{1,j+1}) \}$ the set
of such primes. For condition \eqref{eq:vieE} to hold 
in the current step, it suffices
that any $p\in S_j$ does not divide $s_{2,j}$.
Now, for any such prime $p\in S_j$, again by the recursion 
\eqref{eq:reku} and $(s_{2,j-2},s_{2,j-1})=1$
we see that at most one congruence
class modulo $p$ for $a_{2,j}$ will induce $p|s_{2,j}$. 
Hence by Chinese
Remainder Theorem and estimate \eqref{eq:tima}, there exist
\begin{equation}  \label{eq:latf}
\gg s_{1,j}^{\delta}\cdot \prod_{p\in S_j } (1-p^{-1})  
\end{equation}
many partial quotients $a_{2,j}$ for which $s_{2,j}$ is not divisible 
by any such prime. If this is at least $1$ for large $j$, we are done.
However, again by \eqref{eq:tirE} and a standard estimate for the number of prime divisors of an integer, the cardinality of $S_j$ can be bounded
\[
\sharp S_j \ll \frac{\log (s_{1,j}s_{1,j+1})}{\log\log (s_{1,j}s_{1,j+1})}\ll \log s_{1,j}.
\] 
Hence
we may estimate the latter factor of \eqref{eq:latf}
\[
\prod_{p\in S_j } (1-p^{-1}) < \prod_{n=2}^{ \lfloor\log s_{1,j}\rfloor } (1-n^{-1})= \exp(\sum_{n=2}^{ \lfloor\log s_{1,j}\rfloor } \log(1-n^{-1}) )\ll \exp(-\sum_{n=2}^{ \lfloor\log s_{1,j}\rfloor } n^{-1} )
\]
which gives 
\[
\prod_{p\in S_j } (1-p^{-1})\ll \exp(-\log\log s_{1,j})= (\log s_{1,j})^{-1}.
\]
Hence as $\delta>0$ indeed there remain many suitable $a_{2,j}\in S_j$ in each step so that condition \eqref{eq:vieE} holds as well, and thus \eqref{eq:EE} is true. We have proved the claims of the section.

\begin{remark}  \label{r3}
	The right equality of \eqref{eq:middle} 
	and Lemma~\ref{Mars}, and as it is reasonable to expect that condition \eqref{eq:vieE} is metrically negligible, 
	suggest the bound $n-2+2/(n^2+1)$ stated in Conjecture~\ref{c1}.
	A slightly stronger conjectural bound, still of order $n-2+2n^{-2}-O(n^{-4})$ for large $n$, is motivated in~\S~\ref{sp} below.
\end{remark}

\subsection{Proof of strict inequality in \eqref{eq:theclaim} and lower bounds in \eqref{eq:DREI} } \label{sp}

A small improvement of the bound 
compared to~\S~\ref{s4} 
(and likewise presumably for Conjecture~\ref{c1})
can be made when extending the sets $\mathcal{K}(\tau)\subseteq \R^2$ to larger sets. We then instead of the formula from~\cite{sun}
apply the variational
principle from~\cite{dfsu1, dfsu2} for $m=n=1$ (approximation to a single real number) to estimate their Hausdorff and packing dimension, 
and finally conclude with Theorem~\ref{moschus} and
Lemma~\ref{Mars} again.  We use the template formalism from~\cite{dfsu1, dfsu2}. The general definition of 
templates can be found in~\cite[Defintion~4.1]{dfsu2}, however
in our easiest setting, a template $f=(f_1,f_2)$ consists just two piecewise linear, continuous functions $f_i(t): [0,\infty)\to \R$ with the following properties: $f_1(0)=f_2(0)=0$,
$f_2(t)=-f_1(t)$ and $f_1(t)\le 0\le f_2(t)$ for all $t\ge 0$, 
slopes among $\{-1,0,1\}$, where $0$ is only possible on intervals where $f_1(t)=f_2(t)=0$. Moreover, any local maximum of $f_1$ is a local minimum of $f_2$, hence $f_1(t)=f_2(t)=0$ at such points.

Given $\tau>n$ and $\epsilon>0$, consider the set of $\xi_1\in \R$
that induce a template $f=f_{\tau,\epsilon}=(f_1,f_2)=(f_1, -f_1)$
where the function $f_1(t)$ has periodic pattern as in Figure~1 below. By induce, we mean the exact first successive minimum function $h_1(t)$ (see~\cite[(4.1)]{dfsu2}) has bounded distance from $f_1(t)$ in supremum norm on $[0,\infty)$. 

\begin{tikzpicture}

\draw[->] (0,0) -- (13,0) node[anchor=west] {t};
\draw[thick,->] (0,-3) -- (0,2) node[anchor=south] {$f_1(t)$};

\draw [line width=0.45mm] (2,0) -- (4,-2) ;

\draw [line width=0.45mm] (4,-2) -- (6,0) ;

\draw [line width=0.45mm] (11,0) -- (5.98,0) ;

\fill[black] (2,0) circle (0.06cm) node[anchor=south] {$t_{j}$};

\fill[black] (4,0) circle (0.06cm) node[anchor=south] {$\frac{n\tau+1+\epsilon}{2}t_{j}$};

\fill[black] (6,0) circle (0.06cm) node[anchor=south] {$(n\tau+\epsilon) t_{j}$};

\fill[black] (11,0) circle (0.06cm) node[anchor=south] {$t_{j+1}=\tau^2 t_{j}$};

\node at (4.4,-2.5) { $f_1$};

\node at (5.3,-1.4) { $+1$};

\node at (2.6,-1.4) { $-1$};

\node at (8.3,-0.4) { $0$};

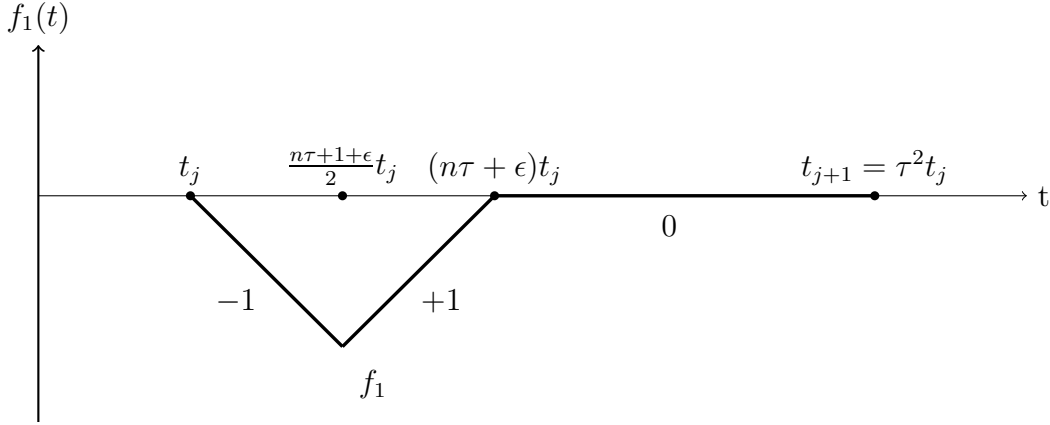
\node at (6,-3.8)  {Figure 1: Sketch period $[t_j,t_{j+1}]$ and slopes of $f_1$ of $1\times 1$-template for $\xi_1$ };

\end{tikzpicture}

Let $\mathcal{S}_{\tau,\epsilon}$ be the set of $\xi_1\in\R$ inducing 
the template $f_{\tau,\epsilon}$ of Figure~1, 
with arbitrary starting point $t_1>0$ (on $[0,t_1]$ we can let
$f_1(t)=f_2(t)=0$).
For any such $\xi_1\in \mathcal{S}_{\tau,\epsilon}$,
by slight abuse of notation,
there is an integer sequence $s_{1,j}\asymp \exp(t_j)$, $j\ge 1$, 
thus in particular
with the property $s_{1,j+1}\asymp s_{1,j}^{\tau^2}$, 
inducing small values $|s_{1,j}\xi_1-r_{1,j}|$ 
(then $r_{1,j}/s_{1,j}$ are convergents to $\xi_1$, but do not comprise all of them) 
that induce the slope $-1$ of the template in the interval
$(t_j, t_j(n\tau+1+\epsilon)/2)$. 
The $s_{1,j}$ depend on the choice of 
$\xi_1\in \mathcal{S}_{\tau,\epsilon}$.
For the moment, take any $\xi_2$ as in~\S~\ref{s4}, i.e. 
so that (all) its
convergent denominators $s_{2,j}$ satisfy \eqref{eq:YE}.
Note that $\xi_2$ does not depend on the concrete
choice of $\xi_1\in \mathcal{S}(\tau,\epsilon)$, and
a template $f^{(2)}=(f_1^{(2)}, f_2^{(2)})$ for $\xi_2$ is given by the $f_i^{(2)}$ 
touching the first axis precisely at positions
$t_j^{(2)}:=\tau t_j$, so that $f_1^{(2)}$ has slope 
$-1$ in $(t_j^{(2)},(t_j^{(2)}+t_{j+1}^{(2)})/2)$ and slope $+1$ in
$((t_j^{(2)}+t_{j+1}^{(2)})/2,t_{j+1}^{(2)})$, and vice versa for $f_2^{(2)}$. Then $t_{j+1}=\tau t_{j}^{(2)}$ holds.
Moreover
\begin{equation} \label{eq:elle}
|s_{1,j} \xi_1-r_{1,j}|\asymp s_{1,j}^{-(n\tau+\epsilon)}, \qquad
|s_{2,j} \xi_2-r_{2,j}|\asymp s_{2,j}^{-\tau^2}
\end{equation}
by the theory of continued fractions.
Again condition
$\tau>n$ is required and sufficient to 
deduce from \eqref{eq:YE} and \eqref{eq:elle} that
$\wo(\xi_1, \xi_2)\ge n+\epsilon/\tau>n$ (with a slight twist if $\tau=\infty$) by considering $\g_j$ and $\h_j$ defined as in \eqref{eq:gjhj}, so that we may apply Theorem~\ref{moschus}.
Assume for the moment the coprimality condition
\eqref{eq:vieE}.
Then again by the same arguments as in~\S~\ref{s4}, for any 
$(\xi_1, \xi_2)$ as above
we get a full measure set $F_n(\xi_1,\xi_2)\subseteq \R^{n-2}$ 
so that for any $\ux=(\xi_1,\ldots,\xi_n)$ with  $(\xi_3,\ldots,\xi_n)\in F_n(\xi_1,\xi_2)$,
the $\g_j$ and $\h_j$ comprise all best approximations. Hence (i)-(v) hold.
Then moreover $\wo(\ux)=\wo(\xi_1,\xi_2)= n+\epsilon/\tau$
and $\om(\ux)=\om(\xi_1,\xi_2)=\tau^2$ by the right formula in \eqref{eq:elle}, however this is not significant 
for the proof. Finally, given $\xi_1\in \mathcal{S}_{\tau,\epsilon}$,
condition \eqref{eq:vieE} can be checked
for certain $\xi_2$ as above by a very similar counting argument 
as in~\S~\ref{s4}, we omit details.

Summing up, any $\xi_1\in \mathcal{S}_{\tau,\epsilon}$ 
induces some $\xi_2$ as above so 
that for a full Lebesgue measure set of
$\xi_3,\ldots,\xi_n$ (depending on $\xi_1,\xi_2$), the vector $\ux=(\xi_1,\ldots,\xi_n)$
lies in $\Theta^n$. Thus again by (I) of Lemma~\ref{Mars} 
applied for $M=\Theta^n$ and $n_1=2, n_2=t=n-2$
and the non-increasing of Hausdorff dimension under projections, we have 
\begin{equation} \label{eq:sho}
\dim_H(\Theta^n)\ge n-2+ \dim_H(\mathcal{S}_{\tau,\epsilon}).
\end{equation}
We evaluate the latter dimension.
The local contraction rate $\delta(f_{\tau,\epsilon},I)$ defined 
in~\cite[Definition~4.5]{dfsu2} in an interval partition of $(0,\infty)$ for the template in Figure~1 is given by
\[
\delta(f_{\tau,\epsilon},I)=\begin{cases}
1, \quad I=(t_j(n\tau+1+\epsilon)/2,\tau^2 t_j) \\
0, \quad I=(t_j,t_j(n\tau+1+\epsilon)/2).
\end{cases}
\]
We evaluate the lower limit $\underline{\delta}(f_{\tau,\epsilon})$ of the average local contraction rate in $t\in [0,T]$ as $T\to\infty$ according to the 
variational principle~\cite[Theorem~4.7]{dfsu2}. As $\epsilon\to 0$,
a short calculation and \eqref{eq:sho} lead for any $\tau>n$ to the bound
\[
\dim_H(\Theta^n)\ge n-2+\lim_{\epsilon\to 0} \dim_H(\mathcal{S}_{\tau,\epsilon})\ge n-2+\lim_{\epsilon\to 0} \underline{\delta}(f_{\tau,\epsilon}) = n-2+ 
\frac{ \tau^2- \frac{n\tau+1}{2} }{(\tau^2-1)\frac{n\tau+1}{2} }.
\]
As $\tau>n$ can be arbitrary, we get
\begin{equation} \label{eq:taurin}
\dim_H(\Theta^n)\ge n-2+ 
\max_{\tau>n} \frac{ \tau^2- \frac{n\tau+1}{2} }{(\tau^2-1)\frac{n\tau+1}{2} }.
\end{equation}
Note that inserting $\tau=n$, we obtain 
the left bound $n-2+1/(n^2+1)$ of \eqref{eq:theclaim}. The maximum over $\tau$
is taken at some slightly larger value, thereby confirming the strict inequality in \eqref{eq:theclaim}.

For the packing dimension, as $\epsilon\to 0$ we evaluate the 
upper limit $\overline{\delta}(f_{\tau,\epsilon})$ of the average contraction rates for the template in Figure~1 as
\[
\lim_{\epsilon\to 0} \overline{\delta}(f_{\tau,\epsilon})=\frac{ \tau^2-\frac{n\tau+1}{2}  }{ \tau^2-\tau }.
\] 
The quantity tends to $1$ as $\tau\to\infty$. 
Thus, for any $\varepsilon>0$ (note $\varepsilon\ne \epsilon$), choosing $\tau$ large enough and $\epsilon>0$ small enough, by
means of (II) from Lemma~\ref{Mars}
and as coordinate projections as Lipschitz maps again do not increase the packing dimension of a set~\cite{falconer},
the variational principle~\cite[Theorem~4.7]{dfsu2} gives the lower
bound 
\[
\dim_P(\Theta^n)\ge n-2+\lim_{\epsilon\to 0} \dim_{P}(\mathcal{S}_{\tau,\epsilon}) \ge  n-2+\lim_{\epsilon\to 0} \overline{\delta}(f_{\tau,\epsilon})>n-1-\varepsilon, \qquad n\ge 2.
\]
As $\varepsilon>0$ can be arbitrarily small, 
we obtain the desired bound $n-1$.

\begin{remark}  \label{hhh}
	As remarked in the proof,
	for $\ux\in\R^{n}$ obtained via $\xi_1, \xi_2$ as constructed above, we have
	\[
	\wo(\ux)= n+\frac{\epsilon}{\tau}, \qquad \om(\ux)=\tau^2.
	\]
	Thus, it is possible to find lower bounds for the Hausdorff
	and packing dimension of sets with a slightly altered definition compared to $\Theta_n(w)$, for example in place of (vi)
	just restricting to $\wo(\ux)=w$ and omitting the claim on the ordinary exponent.
\end{remark}




\section{Sketch of the Proof of Theorem~\ref{thm3}  }

We modify the construction from~\S~\ref{s1},~\ref{lbb}.
Fix $k$ and $n$.
For $\tau>1$ large enough,
we define $k$ sequences $(A_{i,j})_{j\ge 1}$, $1\le i\le k$, of the form
\[
A_{i,j}= p_i^{\alpha_{i,j} }
\]
where $p_i$ is the $i$-th prime number (any set of pairwise coprime integers $\ge 2$ suffices) and $(\alpha_{i,j})_{j\ge 1}$ 
form increasing sequences of positive integers, and for $j\ge 1$ with
the properties
\begin{equation}  \label{eq:gg}
A_{i+1,j}\asymp A_{i,j}^{\tau}\quad (1\leq i\leq k-1), \qquad
A_{1,j+1}\asymp A_{k,j}^{\tau}.
\end{equation}
Again it is clear that such a choice is possible.
Then in particular
\[
A_{1,1} < A_{2,1}< \cdots < A_{k,1} < A_{1,2} < A_{2,2} < \cdots,\qquad
A_{i,j+1}\asymp A_{i,j}^{\tau^{k} }, \quad (1\leq i\leq k,\; j\ge 1).
\]
Let
\[
\xi_i= \sum_{j=1}^{\infty} A_{i,j}^{-1}, \qquad\qquad 1\leq i\le k.
\]
Define the positive integers
\[
F_{i,j}= A_{i,j} \sum_{t=1}^{j} A_{i,t}^{-1} \asymp A_{i,j},\qquad
1\le i \le k,\; j\ge 1,
\]
and derive
\[
\vv_{i,j}= A_{i,j}\cdot \ee_i - F_{i,j}\cdot \ee_{n+1} \in \mathcal{H}_i, \qquad 1\le i\le k,\; j\ge 1.
\] 
Note that all $\vv_{i,j}$ lie in $\mathcal{L}_{n,k}$. Moreover,
it is easy to see that
any $k$ (in fact $k+1$) consecutive $\vv_{i,j}$, ordered by norms $\Vert\hat\vv_{i,j}\Vert=A_{i,j}$, are linearly independent. 
Taking any such collection $\mathcal{C}_{i_0,j_0}$ where $\vv_{i_0,j_0}$ is defined to be its largest vector by norm, then  $\mathcal{L}(i_0,j_0):=\scp{\mathcal{C}_{i_0,j_0}}_{\R}\cap \Z^{n+1}$ is some $k$-dimensional sublattice of $\Z^{n+1}$. More precisely, these 
$\mathcal{L}(i_0,j_0)$ 
are sublattices of $\mathcal{L}_{n,k}=\scp{\ee_{1},\ldots,\ee_k,\ee_{n+1}}_{\Z}$
of codimension $1$. 
Since again $F_{i,j}\equiv 1 \bmod p_i$ for 
any $1\le i\le k$, $j\ge 1$,
it can be shown very similarly to the proof of Case 1 of Theorem~\ref{H} in~\S~\ref{s1}
that $\mathcal{L}(i_0,j_0)=\scp{\mathcal{C}_{i_0,j_0}}_{\Z}$ and the best approximations for $\ux$ within any 
$\mathcal{L}(i_0,j_0)$ must be actually among the 
$k$ vectors $\vv_{i,j}\in \mathcal{C}_{i_0,j_0}$. 

Assume for the moment $n=k$, so that the $\R$-span of any $\mathcal{L}(i_0,j_0)$ above is
a hyperplane of $\R^{n+1}$.
Assume $\tau> \tau_0(k)$ is large enough,
more precisely we require 
\begin{equation} \label{eq:Zzzz}
\frac{\tau^{k}-1}{\tau^{k-1}}>k.
\end{equation}
Then similarly as
in the proof of Case 2 of Theorem~\ref{H} in~\S~\ref{s1}, 
we can further exclude via a generalisation of Lemma~\ref{lemur}  (i.e. no $k+1$ linearly independent vectors of $\Z^{k+1}$ induce approximations of order $k+\epsilon$) again derived from Minkowski's Second Convex Body Theorem applied with $Q=\Vert \vv_{i_0,j_0}\Vert$
that any best approximation of large norm lies outside the sets
$\mathcal{L}(i_0,j_0)$. Hence the $\vv_{i,j}$ comprise all best approximations of large norm. Claim ($ii^{\ast}$) follows directly.
Moreover, since it is easily seen that 
$\vv_{i,j}\notin \mathcal{H}_v$ for any triple 
$(i,j,v)$ with $v\ne i$, no proper subset of
$\{\mathcal{H}_1,\ldots, \mathcal{H}_k\}$ preserves $(ii^{\ast}$). On the other hand, any two-dimensional lattice not contained in some 
$\mathcal{H}_i$ has at most one-dimensional
intersection with each $\mathcal{H}_i$, hence by ($ii^{\ast}$) 
and Remark~\ref{reh} it contains only finitely many best approximations. The first claim of
(vii) follows easily as well. By a similar argument a lattice
of dimension $k$ is insufficient as well: As we find infinitely many
best approximations in each $\mathcal{H}_i$, thus any lattice
containing a tail of best approximations
must have 2-dimensional intersection with each of them, 
by Remark~\ref{reh} again. But this is only possible if
its dimension exceeds $k$.
Moreover, the analogue of \eqref{eq:wdach} that reads
\[
\wo(\xi_1,\ldots,\xi_{k})= \frac{ \tau^{k}-1}{\tau}, \qquad
\om(\xi_1,\ldots,\xi_{k})=\tau^{k}-1
\]
holds. Since $k=n$, we have found a suitable vector $(\xi_1,\ldots,\xi_n)=(\xi_1,\ldots,\xi_{k})$. 

Now assume $n>k$. Increase $\tau$ if necessary, so that in
addition to \eqref{eq:Zzzz}, we have $(\tau^{k}-1)/\tau>n$
as well. Take any $(\xi_1,\ldots,\xi_{k})$ as 
in the case $n=k$ above. We see
that 
\[
\wo(\xi_1,\ldots,\xi_{k})=\frac{\tau^{k}-1}{\tau}>n.
\]
Hence
we may apply a more general version of Theorem~\ref{moschus} from~\cite{ngm}, stating that for almost all vectors $(\xi_{k+1}, \xi_{k+2},\ldots,\xi_{n})\in\R^{n-k}$
with respect to Lebesgue measure, the magnified
vector $\ux=(\xi_1,\ldots,\xi_n)$ has
the same tail of best approximations up to lifting, i.e.
again lie in the according lattices $\mathcal{H}_i=\mathcal{H}_i(n)$. The claims ($ii^{\ast}$), (vii) follow analogously to the case
$n=k$. The lower bound $n-k$ for the Hausdorff
dimension is clear, finally the inequality being strict can be shown by the method from~\S~\ref{s4}, we omit details.


\section{Proof of Theorem~\ref{neu} }

We first point out that
in~\S~\ref{s1} above, essentially
we only require the special form of $\xi_1, \xi_2$ to  
exclude in Case 1 other best approximations within the lattices $\scp{\vv_j,\w_j}_{\Z}$ and $\scp{\w_j,\vv_{j+1}}_{\Z}$, sublattices 
of $\scp{\ee_1,\ee_2,\ee_{n+1}}_{\Z}$. In context of Theorem~\ref{neu},
this will not be an issue when we simply lift the sequence of best
approximations of some $(\xi_1,\xi_2)$ to $\Z^{n+1}$.
However, it turns out that our method below requires the property $\wo(\xi_1,\xi_2)> 3n-4$ to eliminate the dependence of 
the derived set $F_n$ from $\xi_1,\xi_2$. 
So let $\xi_1, \xi_2$ be any such numbers, totally irrational.
To simplify the argument below, we fix any constant
\begin{equation} \label{eq:israel}
\sigma\in(3n-4,\wo(\xi_1,\xi_2)).
\end{equation}
Let $(v_{j,1}, v_{j,2}, v_{j,3})\in\Z^3$ be the sequence of best approximations associated to $(\xi_1,\xi_2)\in\R^2$.
Via embedding into $\Z^{n+1}$ it gives rise to a sequence
which by abuse of notation we want to denote again 
by $(\vv_j)_{j\ge 1}$, so
\[
\vv_j= v_{j,1}\ee_1 + v_{j,2}\ee_2 + v_{j,3}\ee_{n+1} \in\Z^{n+1}, \qquad j\ge 1.
\]
Note that any $\vv_j$ lies in the fixed
three-dimensional sublattice $\scp{\ee_1, \ee_2, \ee_{n+1} }_{\Z}$ of $\Z^{n+1}$ as in the definition of $\widetilde{\Gamma}_n$.

Fix any $(\xi_3,\ldots, \xi_n)\in \R^{n-2}$ as in $\mathcal{T}_n$, that
is any vector satisfying the generic property
\begin{equation} \label{eq:ome}
\om(\xi_3,\ldots, \xi_n)=n-2. 
\end{equation}
Here we cannot use Theorem~\ref{moschus}, so we follow another strategy.
Again assume $\bb$ is any best approximation of large norm. 
There is a unique $j$ such that 
\begin{equation} \label{eq:Tzz}
\Vert \hat\vv_j\Vert\le \Vert \hat\bb\Vert< \Vert \hat\vv_{j+1}\Vert.
\end{equation}
We will show that $\bb=\vv_j$ to finish the proof.
First observe that since $\bb$ is a best approximation
of norm at least $\Vert\hat\vv_j\Vert$, we know that
\begin{equation}  \label{eq:seezz}
|\bb\cdot \uxs|\le |\vv_j\cdot \uxs|.
\end{equation}
Besides, a simple application of Minkowski's Second Convex Body Theeorem implies

\begin{lemma}  \label{llama}
	Let $\xi_3,\ldots,\xi_{n}$ satisfy \eqref{eq:ome}
	and $\varepsilon>0$.
	For any $Q\ge Q_0(\varepsilon)$, there exist $n-1$ linearly independent integer vector solutions $\uu_1, \ldots, \uu_{n-1}$ 
	in $\Z^{n-1}$ to
	\[
	\max_{1\le i\le n-1} \Vert \hat\uu_i\Vert\le Q, \qquad 
	| \uu_i\cdot (\xi_3,\ldots,\xi_{n},1)|\le Q^{-(n-2)+\varepsilon},
	\]
	where again $\hat\uu_i$ omits the last coordinate of $\uu_i$.
\end{lemma}

\begin{remark}  \label{hirsch}
	If we would restrict to $c$-badly approximable vectors in $\R^{n-2}$
	whose Hausdorff dimension tends to $n-2$ as $c\to 0^{+}$, then we could sharpen the right hand side estimate to $\ll_{c,n} Q^{-(n-2)}$,
	which would simplify the proof below a little.
\end{remark}

\begin{proof}
	Consider the integer lattice in $\Z^{n-1}$ and for $Q\ge 1$ the convex body
	consisting of $(x_1,\ldots,x_{n-1})\in\R^{n-1}$ with 
	\[
	\max_{1\le i\le n-2} |x_i|\le Q,\quad  |\xi_3 x_1+\xi_4 x_2+\cdots +\xi_{n} x_{n-2} + x_{n-1}|\leq Q^{-(n-2)}.
	\]
	The volume is $2^{n-1}$, so by Minkowski's Second convex body theorem the product of the successive minima are of order $\asymp_n 1$ as well.
	The condition \eqref{eq:ome} tells us that the first successive minimum
	is $\gg Q^{-\epsilon}$ for any $\epsilon>0$ and all $Q\ge Q_0(\epsilon)$. Hence all successive minima
	are $\ll_n Q^{\epsilon/(n-2)}$. This easily yields the claim by slightly
	modifying $\epsilon$ to some $\varepsilon$ if necessary. 
\end{proof}

Let $\uu_1,\ldots,\uu_{n-1}$ be the vectors
from Lemma~\ref{llama} for the parameter
\begin{equation} \label{eq:kuh}
Q=\Vert \hat\vv_{j+1}\Vert^{  \frac{\sigma+1}{n-1}},
\end{equation}
that later turns out optimal, and let
$\yy_1,\ldots,\yy_{n-1}$ be their embeddings
into $\Z^{n+1}$ via 
\[
\yy_i=(0,0,\uu_i), \qquad 1\le i\le n-1.
\]
Let
\[
Y=\{ \yy_1,\ldots,\yy_{n-1}\}, \qquad \scp{Y}_{\R}=\scp{\ee_3,\ldots,\ee_{n+1} }_{\R}.
\]
Assume we have already shown that
\begin{equation}  \label{eq:Zy}
\scp{ \vv_j, \vv_{j+1}}_{\R}\cap \scp{Y}_{\R}= \scp{ \vv_j, \vv_{j+1}}_{\R}\cap \scp{\ee_{n+1}}_{\R}=\0, \qquad j\ge j_0.
\end{equation}
Then the vectors $\{ \yy_1,\ldots,\yy_{n-1}, \vv_j, \vv_{j+1} \}$ span the entire space $\R^{n+1}$ for any $j\ge j_0$,
i.e. 
\begin{equation}  \label{eq:together}
\scp{ Y,\vv_{j}, \vv_{j+1}}_{\R} = \R^{n+1}.
\end{equation}
If we assume that $\bb\in \scp{\vv_j,\vv_{j+1}}_{\R}\cap \Z^{n+1}$, 
then obviously $\bb\in \scp{\ee_1,\ee_2,\ee_{n+1}}_{\Z}$. But 
then if we write $\bb=(b_1,b_2,0,\ldots,0,b_{n+1})$, the vector $(b_1,b_2,b_{n+1})\in\Z^3$ constitutes a best approximation with respect to $(\xi_1,\xi_2)$ by \eqref{eq:Tzz} and \eqref{eq:seezz}. 
In fact $(b_1,b_2,b_{n+1})=(v_{j,1}, v_{j,2}, v_{j,3})$ since
the latter comprise all ordered 
best approximations and $(\xi_1,\xi_2)$ is totally irrational,
hence equivalently indeed $\bb=\vv_j$.
Assume otherwise that
\begin{equation}  \label{eq:beraber}
\dim(\scp{\vv_j,\vv_{j+1},\bb}_{\R})= 3,
\end{equation}
so $\vv_j,\vv_{j+1},\bb$ span a space of dimension
three. We will derive a contradiction
to \eqref{eq:Tzz}, \eqref{eq:seezz} by application of Minkowski's Second Lattice Body Theorem.

With $Q$ as in \eqref{eq:kuh}, let
\[
T=\Vert \hat\vv_{j+1}\Vert^{  \frac{\sigma+1}{n+1}}=Q^{  \frac{n-1}{n+1}}.
\]
Then by construction and \eqref{eq:Tzz} we have
\begin{equation} \label{eq:FFF}
\max\{ \Vert \hat\bb\Vert, \Vert \hat\vv_j\Vert, \Vert \hat\vv_{j+1}\Vert\}= \Vert \hat\vv_{j+1}\Vert=T^{\frac{n+1}{\sigma+1} } 
\end{equation}
and by \eqref{eq:seezz} and since $\vv_j, \vv_{j+1}$ are two
successive best approximations and by \eqref{eq:israel} we have
\begin{equation} \label{eq:HHH}
\max\{ |\bb\cdot \uxs|,\; |\vv_j\cdot \uxs|, \; |\vv_{j+1}\cdot \uxs| \}
= |\vv_j\cdot \uxs|\ll \Vert \hat\vv_{j+1}\Vert^{-\sigma}= T^{-\sigma\frac{n+1}{\sigma+1} }.
\end{equation}
The definition of the $\yy_i$ yields
\begin{equation}  \label{eq:JJ}
\max_{1\le i\le n-1} \Vert \hat\yy_i\Vert =\max_{1\le i\le n-1} \Vert \hat\uu_i\Vert \le Q=  T^{  \frac{n+1}{n-1}}
\end{equation}
and
\begin{equation}  \label{eq:GGG}
\max_{1\le i\le n-2} |\yy_i\cdot \uxs| =\max_{1\le i\le n-2} |\uu_i\cdot (\xi_3,\ldots,\xi_n,1)|  
\le Q^{-(n-2)+\varepsilon}=
T^{(-(n-2)+\varepsilon)  \frac{n+1}{n-1}}.
\end{equation}
Define the convex body $K(T)$ of $(x_1,\ldots,x_{n+1})\in\R^{n+1}$ with 
\[
\max_{1\le i\le n} |x_i|\le T, \qquad |x_1 \xi_1+ \cdots +x_n \xi_n+x_{n+1}| \le T^{-n},
\]
of volume $2^{n+1}$ regardless of $T$.
Consider the successive minima $\lambda_i(K(T),\Z^{n+1})$, $1\le i\le n+1$, with respect to $K(T)$ and $\Z^{n+1}$.
One verifies by \eqref{eq:FFF}, \eqref{eq:HHH} and by \eqref{eq:beraber} with some calculation 
that the first three successive minima are bounded
from above by
\begin{equation} \label{eq:edrei}
\lambda_i(K(T),\Z^{n+1})\ll T^{\frac{n-\sigma}{1+\sigma}}, \qquad 1\le i\le 3.
\end{equation}
Indeed, for $c>0$, the first $n$ coordinates of the box $cT^{(n-\sigma)/(1+\sigma)}\cdot K(T)$
are bounded by
\[
T\cdot cT^{\frac{n-\sigma}{1+\sigma} }=  cT^{\frac{n+1}{\sigma+1}}
\]
and the last coordinate by
\[
T^{-n}\cdot cT^{\frac{n-\sigma}{1+\sigma}}= cT^{-\frac{(n+1)\sigma}{\sigma+1}}
\]
so for large enough $c$ the linearly independent points $\bb, \vv_j, \vv_{j+1}$ belong to this box.
Similarly, by \eqref{eq:JJ}, \eqref{eq:GGG} and \eqref{eq:together}
for larger indices we get
\begin{equation}  \label{eq:dree}
\lambda_i(K(T),\Z^{n+1})\le T^{\frac{2}{n-1}+3\varepsilon}, \qquad 4\le i\le n+1.
\end{equation}
Indeed, we see that the first $n$ coordinates of the box $T^{2/(n-1)+3\varepsilon} K(T)$ are 
bounded by
\[
T\cdot T^{\frac{2}{n-1}+3\varepsilon}\ge T\cdot T^{\frac{2}{n-1}}= T^{\frac{n+1}{n-1}}
\]
and the last by
\[
T^{-n}\cdot T^{\frac{2}{n-1}+3\varepsilon}= T^{ -\frac{(n-2)(n+1)}{n-1}+ 3\varepsilon }\ge T^{ -\frac{(n-2)(n+1)}{n-1}+ \varepsilon \frac{n+1}{n-1} }, \qquad n\ge 2,
\]
and we recognize the right hand side as the value in \eqref{eq:GGG}. Hence 
all $\yy_i$ lie in this box.

Combining \eqref{eq:edrei} and \eqref{eq:dree}, we get 
that the product of all successive minima is bounded by
\[
\prod_{i=1}^{n+1} \lambda_i(K(T),\Z^{n+1})\ll T^{ \frac{3(n-\sigma)}{1+\sigma}+ \frac{2(n-2)}{n-1}+3(n-2)\varepsilon}.
\]
We see that for $\varepsilon>0$ small enough, the exponent of $T$ is negative since $\sigma>3n-4$. Finally for large $T$ this contradicts Minkowski's Second Convex Body Theorem that predicts 
\[
\prod_{i=1}^{n+1} \lambda_i(K(T),\Z^{n+1})\gg_n 1
\]
with implied constant 
independent of $T$. We deduce that the vectors
$\pm \vv_j$ are the only
best approximations of large norm, which lie 
in $\scp{\ee_1, \ee_2, \ee_{n+1}}_{\Z}$. Hence indeed any $\ux$ as above
lies in $\widetilde{\Gamma}_n\subseteq \Gamma_n$, and $\wo(\ux)=\wo(\xi_1,\xi_2)$.

We finally verify \eqref{eq:Zy}. The left identity is clear by the 
form of the $\yy_i$ and $\vv_j$. The right
is clearly true if $v_{j,1}v_{j+1,2}-v_{j+1,1}v_{j,2}\ne 0$ for any large $j$, as then
$a\vv_j+b\vv_{j+1}\in \scp{\ee_{n+1}}$ for real numbers $a,b$ 
implies $a=b=0$.
Assume otherwise for some large $j$ the expression 
vanishes. Then for some $p/q\in\mathbb{Q}$, with $q\ge 1$ and
in reduced form, 
we have the vector identity
\[
(p/q)\cdot (v_{j,1},v_{j,2})= (v_{j+1,1}, v_{j+1,2})
\]
in $\Z^2$. 
Define the real numbers $\sigma_j$ via
\begin{equation} \label{eq:ivo}
|v_{j,1}\xi_1 + v_{j,2}\xi_2 + v_{j,3}|= \max\{ |v_{j,1}|, |v_{j,2}| \}^{-\sigma_j}.
\end{equation}
Then clearly $\sigma_j>3n-4$ for large $j$ by \eqref{eq:israel}.
Since any best approximation is primitive we must have
$(p/q)v_{j,3}\notin \mathbb{Z}$, in particular $(p/q)v_{j,3}\ne v_{j+1,3}$.
We may write
\begin{align*}
v_{j+1,1}\xi_1 + v_{j+1,2}\xi_2 + v_{j+1,3}=(p/q)(v_{j,1}\xi_1 + v_{j,2}\xi_2 + v_{j,3})+ v_{j+1,3}-(p/q)v_{j,3}.
\end{align*}
Since $q\nmid v_{j,3}$, the rational number $v_{j+1,3}-(p/q)v_{j,3}$
has absolute value at least $q^{-1}$. Thus
as $q$ divides both $v_{j,1}$ and $v_{j,2}$ and hence
$q\le \max\{ |v_{j,1}|, |v_{j,2}| \}$, by triangle inequality we infer 
\begin{align} 
|v_{j+1,1}\xi_1 + v_{j+1,2}\xi_2 + v_{j+1,3}|
&\ge |v_{j+1,3}-(p/q)v_{j,3}|- |p/q|\cdot |v_{j,1}\xi_1 + v_{j,2}\xi_2 + v_{j,3}| \nonumber \\ 
&\ge 
q^{-1}- |p/q|\cdot|v_{j,1}\xi_1 + v_{j,2}\xi_2 + v_{j,3}| \nonumber
\\ &\ge \max\{ |v_{j,1}|, |v_{j,2}| \}^{-1}-|p/q|\cdot \max\{ |v_{j,1}|, |v_{j,2}| \}^{-\sigma_j}.  \label{eq:gans}
\end{align}
Now since $(v_{j,1}, v_{j,2}, v_{j,3})$ and $(v_{j+1,1}, v_{j+1,2}, v_{j+1,3})$ are two consecutive best approximations, in view of
$\wo(\xi_1,\xi_2)>\sigma$ and \eqref{eq:ivo} it is not hard to see that
\[
|p/q|=\frac{\max\{ |v_{j+1,1}|, |v_{j+1,2}| \} }{ \max\{ |v_{j,1}|, |v_{j,2}| \} }\le \max\{ |v_{j,1}|, |v_{j,2}|\}^{\sigma_j/\sigma-1}.
\]
Hence, since obviously by $\sigma>3n-4>1$ we have
\[
\sigma_j - \sigma_j/\sigma+1> 1,
\]
inserting in \eqref{eq:gans} we see that 
\begin{align*}
|v_{j+1,1}\xi_1 + v_{j+1,2}\xi_2 + v_{j+1,3}|\gg \max\{ |v_{j,1}|, |v_{j,2}| \}^{-1}\ge \max\{ |v_{j+1,1}|, |v_{j+1,2}| \}^{-1}.
\end{align*}
This clearly contradicts $(v_{j+1,1}, v_{j+1,2}, v_{j+1,3})$ being a best approximation by Dirichlet's Theorem \eqref{eq:diri} and as $n\ge 2$, for large $j$. This completes
the proof of \eqref{eq:Zy}.


\begin{remark}  \label{r9}
	The proof works similarly with a bit more effort for $(\xi_1,\ldots,\xi_n)\in\R^{n}$ any vector
	with $\wo(\xi_1,\xi_2)> \max\{ \om(\xi_3,\ldots,\xi_n), 3n-4 \}$.
\end{remark}

\vspace{0.5cm} {\em The author thanks Nikolay Moshchevitin for bringing
to my attention the paper by Kolomeikina, Moshchevitin related to Problem~1 and providing explanations.}

\end{document}